\documentclass[english]{jnsao}
\usepackage[utf8]{inputenc}
\usepackage{cleveref}
\usepackage{enumitem}

\newtheorem{properties}[theorem]{Properties}

\renewcommand{\epsilon}{\varepsilon}
\newcommand{\x}{x_1, x_2}

\newcommand{\eu}{\mathrm{e}}

\newcommand{\vel}{v}
\newcommand{\vela}{u}
\newcommand{\veli}{b}

\AddToHook{env/proposition/begin}{\crefalias{theorem}{proposition}}
\AddToHook{env/lemma/begin}{\crefalias{theorem}{lemma}}
\AddToHook{env/corollary/begin}{\crefalias{theorem}{corollary}}
\AddToHook{env/definition/begin}{\crefalias{theorem}{definition}}
\AddToHook{env/assumption/begin}{\crefalias{theorem}{assumption}}
\AddToHook{env/remark/begin}{\crefalias{theorem}{remark}}
\AddToHook{env/algorithm/begin}{\crefalias{theorem}{algorithm}}

\title{Well-posedness of an optical flow based optimal control formulation for image registration}
\author{Johannes Haubner\thanks{Department of Mathematics and Scientific Computing, NAWI Graz, University of Graz, Austria \email{(johannes.haubner@uni-graz.at) }}
    \and Christian Clason\thanks{Department of Mathematics and Scientific Computing, NAWI Graz, University of Graz, Austria; BioTechMed-Graz, Graz, Austria \email{(c.clason@uni-graz.at)}}
}
\acknowledgements{This research was funded in whole or in part by the Austrian Science Fund (FWF) 10.55776/F100800.}
\shortauthor{Haubner, Clason}
\manuscripteprinttype{arxiv}
\manuscripteprint{2507.10188}

\begin{document}
\maketitle

\begin{abstract}
    We consider image registration as an optimal control problem using an optical flow formulation, i.e., we discuss an optimization problem that is governed by a linear hyperbolic transport equation. Requiring Lipschitz continuity of the vector fields that parametrize the transformation leads to an optimization problem in a non-reflexive Banach space. We introduce relaxations of the optimization problem involving smoothed maximum and minimum functions and appropriate Orlicz spaces. To derive well-posedness results for the relaxed optimization problem, we revisit and establish new existence and uniqueness results for the linear hyperbolic transport equations. We further discuss limit considerations with respect to the relaxation parameter and discretizations.
\end{abstract}

\section{Introduction}

Modeling optimization problems in non-reflexive Banach spaces has demonstrated its utility in several applications in science and engineering, such as image processing. Its use comes from the property that the total variation term in the $\mathrm{BV}$-norm promotes piecewise constant behavior of optimal solutions \cite{bredies2016pointwise, ring2000structural}. Optimization in non-reflexive Banach spaces is also appealing in shape optimization \cite{Deckelnick2022}. In shape optimization via the method of mappings, it is desirable to work with transformations that are bi-Lipschitz, i.e.~are bijective and both the transformation and its inverse attain $W^{1,\infty}$-regularity \cite{brandenburg2009continuous}. In this work, we focus on image registration, where similar requirements on the set of admissible transformations are present as for shape optimization, cf.~\cite{zapf2025medical}. Current Hilbert space approaches either work with too smooth function spaces that embed into $C^1$ and, therefore, e.g., prevent the appearance of non-smooth features, or in too weak function space settings that do not provide theoretical guarantees.

The goal of this work is to provide a theoretical foundation to bridge this gap. To this end, we consider
an optimal control problem that is governed by a linear hyperbolic transport equation, see e.g., \cite{borzi2003optimal, chen2011image, zapf2025medical}, given by
\begin{equation}
    \begin{aligned}
        \min_{v \in V_{\mathrm{ad}}} &J(\phi(\cdot, T), \phi_{\mathrm{tar}}) + \mathcal R(v) \\
        &\begin{aligned}
            \text{s.t. }  \partial_t \phi + \vel \cdot \nabla \phi &= 0 && \text{in } (0,T) \times \Omega, \\
            \phi(\cdot, 0) &= \phi_0 && \text{on } \Omega,
        \end{aligned}
    \end{aligned}
    \label{opt::prob}
\end{equation}
where $\Omega \subset \mathbb R^d$, $d \in \lbrace 2, 3 \rbrace$, denotes a bounded Lipschitz domain, $T > 0$, and $\phi_0, \phi_{\mathrm{tar}}: \Omega \to \mathbb R$ denote the input and target image. Further,  $V_{{\mathrm{ad}}}$ denotes the set of admissible transformations, $\vel$ denotes a vector field, and the objective function consists of a discrepancy measure $J$ and a regularization term $\mathcal R(\vel)$. Note that we write $\vel$ for stationary vector fields (used in the scope of analyzing the optimization problem) and $b$ for time-dependent vector fields (used in the scope of analyzing the transport equation). We are particularly interested in settings that involve an $W^{1,\infty}$-regularization of the velocity in space. For the sake of clarity and to present our main ideas, we discuss one realization with stationary vector fields in detail, which works with $V_{\mathrm{ad}} = W_0^{1,\infty}(\Omega)^d \cap H^{1+\sigma}(\Omega)^d$ and
\begin{equation*}
    \mathcal R(v) \coloneqq \beta \| v \|_{W^{1,\infty}_0(\Omega)^d} + \frac\alpha2 \| v \|_{H^{1+\sigma}(\Omega)^d}^2
\end{equation*}
for $\alpha, \beta > 0$ and $\sigma \in (0,\frac12)$ (see \eqref{w1inftynorm} for the definition of $\|\cdot\|_{W^{1,\infty}_0(\Omega)^d}$). While restricting to stationary vector fields for the specific application, we derive a general framework for the optimization and discuss the linear hyperbolic transport equation in a setting that allows for instationary vector fields.

Optimization problems with a $W^{1, \infty}$-norm regularization term are challenging due to the non-differentiability of its norm, see \cite{clason2012minimum}.
A common strategy, which is also called
``augmented Morozov regularization'' \cite{clason2012fitting}, is to introduce an additional scalar optimization variable $t$ that replaces the regularization in the objective function and to add constraints that bound the components of the regularization pointwise almost everywhere.
Instead, we consider a relaxation of an $L^\infty$-seminorm involving an approximation of the maximum and minimum function, see also \cite{bertazzoni2024approximation}, or \cite{Kruse2015} for interior point methods in function space, and replace the regularization term $\mathcal R$ by a parameterized regularization term $\mathcal R_\gamma$ given by
\begin{equation*}
    \mathcal R_\gamma(v) \coloneqq \beta \Psi_\gamma(v) + \frac\alpha2 \| v\|_{H^{1+\sigma}(\Omega)^d}^2,
\end{equation*}
where $\Psi_\gamma$ is defined in \eqref{psigamma}.
The resulting smooth approximation for an $L^\infty$-seminorm is given by
\begin{equation}
    \chi_\gamma(\vela) \coloneqq \gamma \ln\left(\frac{1}{|\Omega|}\int_{\Omega} \exp({\gamma^{-1} \vela(x)}) \mathrm{d} x\right) + \gamma \ln\left(\frac{1}{|\Omega|}\int_\Omega \exp({-\gamma^{-1} \vela(x)}) \mathrm{d}x\right),
    \label{formula_chigamma}
\end{equation}
which is naturally posed in the Orlicz space $L_{\exp}(\Omega)$. It is related to the entropic risk measure,
allowing us to use several of its properties for developing a concise framework.
Moreover, Orlicz spaces play a critical role in statistics \cite{Buldygin2000}, as they allow for modeling of precise tail behaviors of random variables and regularity of stochastic processes. The relaxation of the regularization leads to optimization in a non-reflexive Banach space with $\mathrm{div}(\vel) \in L_{\exp}(\Omega)$, for which it is not ensured that $\mathrm{div}(\vel) \in L^\infty(\Omega)$. This prevents the application of existence results for transport equations established in \cite{jarde2019existence}.

The optimal control problem \eqref{opt::prob} with instationary vector fields $\veli$ was theoretically investigated in \cite{diperna1989ordinary, ambrosio2004transport, crippa2008flow, crippa2014initial, crippa2014note, de2008ordinary, jarde2018analysis, jarde2019existence}.
Choosing vector fields in $L^1((0,T),W^{1,\infty}_0(\Omega)^d)$ ensures that the underlying transformation which maps $\phi_a$ to $\phi(\cdot,T)$ is bi-Lipschitz \cite[Thm.~5.1.8]{jarde2018analysis}, which is desirable for image registration. Moreover, the existence of solutions to the optimal control problem \eqref{opt::prob} can be shown under mild regularity assumptions. For example, \cite{jarde2019existence} establishes existence results under the regularity requirements $\veli \in L^\infty((0,T) \times \Omega)^d \cap L^2((0,T), \mathrm{BV} (\Omega)^d)$ and $\mathrm{div}(\veli) \in L^2((0,T), L^\infty(\Omega))$, which is shown by generalizing results from \cite{ambrosio2004transport, crippa2008flow, crippa2014initial, crippa2014note, de2008ordinary}, by establishing novel stability results \cite[Thm.~5.1 and Thm.~5.2]{jarde2019existence}, and by modeling the set of admissible controls adequately \cite[Sec.~8.2]{jarde2019existence}.  Up until now, there are several results on well-posedness of the optimization problem that are constrained by linear hyperbolic transport equations \cite{chen2011image, jarde2018analysis}. However, these results do not cover the case where the divergence of the vector field $\veli$ is only ensured to be in $L^p((0,T), L_{\exp}(\Omega))$, where $T >0$ and $\Omega \subset \mathbb R^d$, $d \geq 1$, is an appropriate Lipschitz domain. Moreover, this setting is not covered by uniqueness results that are available for the linear hyperbolic transport equations. While the choice $\mathrm{div}(\veli) \in L^1((0,T), L^\infty(\Omega))$ is the basis for the theoretical results in, e.g., \cite{crippa2008flow, crippa2014initial, diperna1989ordinary, de2008ordinary, jarde2018analysis, jarde2019existence}, several generalizations have been considered in the literature. On the one hand, one-sided Lipschitz conditions have been considered and the concept of duality solutions has been introduced \cite{conway1967generalized, bouchut2005uniqueness, lions2023transport}. However, this condition only provides uniqueness of the forward transport equation. Moreover, $\mathrm{div}(\veli) \in L^1((0,T), \mathrm{BMO}(\Omega))$ has been considered as a natural extension \cite{subko2014remark, mucha2010transport}. In \cite{ambrosio2009flows}, uniqueness results for $\exp^{c [\mathrm{div}(\veli)]^-} \in L^\infty((0,T), L^1(\Omega))$ for $[\cdot]^-(x) = \max(0, -(\cdot)(x))$ have been derived. We work with the setting $\mathrm{div}(\veli) \in L^p((0,T),  L_{\exp}(\Omega))$, $p \in (1,\infty]$, which is more general than the above results ($L^\infty$ in time is more restrictive in time than $L^p$, and, e.g. on cubes, $L_{\exp}$ is a larger class of functions than $\mathrm{BMO}$ \cite[Prop.~2]{tao2021survey}).

The relaxed optimization problem, for which $L^\infty$-seminorms are approximated by \eqref{formula_chigamma} and which is posed in a Banach space that we denote by $W^{1,\exp}$, has still a regularization term that is not differentiable. Also deriving differentiability results for the transport equation with respect to the vector field is involved and typically based on the notion of measure solutions \cite{diff_measuresol, Diehl2024, jarde2018analysis}, under additional restrictions and assumptions on the objective function $J$ in \eqref{opt::prob}, and the regularity of the vector fields, e.g., \cite{jarde2018analysis} .
In the scope of this work, we focus on eliminating the non-differentiability coming from the regularization term. Discretization of the problem yields a canonical way to obtain differentiability of $\Psi_\gamma$ due to the additional $W^{1,\infty}$-regularity that comes from the discretization.

\subsection*{Overview}
This work is organized as follows. \Cref{sec::general_framework} introduces the conceptual framework. Within this conceptual framework, we discuss well-posedness and limit considerations. The rest of the work shows that image registration with $W^{1,\infty}$-regularization can be considered in that framework.
In \cref{sec::preliminaries}, we recall the definition of the Orlicz spaces $L_{\exp}$ and $L{\log}L$ and collect results that are needed for the analysis. In \cref{sec::approxnorm}, we consider the approximation to a $L^\infty$-seminorm. We derive properties on boundedness, sequential weak$^\ast$ lower semicontinuity, and Fr\'echet differentiability. In \cref{sec::transporteq}, we consider the transport equation, recall known results from the literature \cref{subsec::existence}, prove a novel uniqueness result (see \cref{subsec::uniqueness}), and show the validity of a stability result from the literature under these requirements (see \cref{subsec::stability}). In \cref{sec::wellposed}, we collect the results from the previous sections in order to show that the properties of the conceptual framework are fulfilled. We finally show that there exists a sequence of semidiscretized relaxed optimization problems with differentiable regularization terms that approximate the original optimization problem \eqref{opt::prob} with a $W^{1,\infty}$-regularization. The appendix contains ancillary proofs.

\section{Conceptual framework}
\label{sec::general_framework}

We consider the following conceptual framework and justify in \cref{lemma::imreg::prop,p16} that the motivating image registration example can be considered in this framework. The abstract framework allows to extend the analysis in this work to other settings that fulfill the properties \ref{assumption:1}--\ref{assumption:limgamma}.

\subsection{Well-posedness}

We establish a well-posedness result under the following properties:

\begin{properties}[{control space, feasible set, convergence compatibility}]~
    \begin{enumerate}[label=(\textsc{p}\arabic*)]
        \item The space $X$ is a Hilbert space, $W$ is a separable Banach space, and the embedding $X \hookrightarrow W$ is compact.\label{assumption:1}
        \item The space $Y$ is the dual to a separable Banach space.
            \label{assumption:1.1}
        \item The embedding $Y \hookrightarrow W$ is continuous. \label{assumption:1.2}
        \item For each sequence $(v^k)_{k \in \mathbb N} \subset X \cap Y$ that converges weakly to $\hat v$ in $X$ and weakly$^\ast$ to $\bar v$ in $Y$, it holds that $\hat v = \bar v$ in $X \cap Y$.
            \label{limit.assumption}
    \end{enumerate}
\end{properties}

Now, we formulate conditions on the composite objective functions.

\begin{properties}[{composite objective function: continuity and boundedness}]~
    \begin{enumerate}[label=(\textsc{p}\arabic*)]
        \setcounter{enumi}{4}
        \item For each $\gamma >0$, $\psi_\gamma: Y \to \mathbb R\cup \lbrace \infty \rbrace$ is weakly$^\ast$ sequentially  lower semicontinuous.
            \label{assumptions:weakastconv}
        \item The function $\eta: X \to \mathbb R$ is weakly lower semicontinuous.
            \label{assumptions:weakconv}
        \item The function $j: Y \to \mathbb R$ is strongly continuous w.r.t.~the $W$-topology, i.e., if $(\vel^k)_{k \in \mathbb N} \subset Y$ and $\vel \in Y$ with $\vel^k \to \vel$ as $k\to \infty$ w.r.t. the strong topology in $W$, then $j(\vel^k) \to j(\vel)$ as $k\to \infty$. \label{assumption:4}
        \item The functions $j$, $\psi_\gamma$
            ($\gamma > 0$), and $\eta$ are bounded from below on $X \cap Y$. \label{assumption:13}
        \item The boundedness of $(\psi_\gamma(\vel^k) + \eta(\vel^k))_{k \in \mathbb{N}}$ implies boundedness of $(\| \vel^k\|_Y + \|\vel^k\|_X)_{k \in \mathbb{N}}$.
            \label{assumption:14}
    \end{enumerate}
\end{properties}

For this general setting, we consider the optimization problem with $\gamma > 0$,
\begin{equation}
    \min_{\vel \in X\cap Y} j(\vel) + \psi_\gamma(\vel) + \eta(\vel).
    \label{general:opt:prob::}
\end{equation}

Using the direct method of calculus of variations, the following theorem shows well-posedness of the optimization problems \eqref{general:opt:prob::}.

\begin{lemma} Let properties \ref{assumption:1}--\ref{assumption:14} hold. Then \eqref{general:opt:prob::} admits a solution.
    \label{theorem::full}
\end{lemma}

\begin{proof} Due to \ref{assumption:13} we know that the objective function $f(\vel)\coloneqq j(\vel) + \psi_\gamma(\vel) + \eta(\vel)$ of \eqref{general:opt:prob::} is bounded from below. Let $\bar f \coloneqq \inf_{\vel \in X \cap Y} f(\vel)$. We consider a minimizing sequence $(\vel^k)_{k \in \mathbb N}$ such that $f(\vel^k) \leq \bar f + \frac1k$. Hence, there exists $c > 0$ such that $\psi_\gamma(\vel^k) + \eta(\vel^k) \leq c$ for all $k \in \mathbb N$. Due to \ref{assumption:14}, $(\|\vel^k\|_Y + \| \vel^k\|_X)_{k\in \mathbb{N}}$ is bounded. Hence, using \ref{assumption:1}--\ref{limit.assumption}, there exists a subsequence $(\vel^k)_{k \in K}$, $K \subset \mathbb N$, and $\vel \in X \cap Y$ such that $\vel^k \rightharpoonup \vel$ w.r.t.~the weak topology in $X$, $\vel^k \rightharpoonup^\ast \vel$ w.r.t.~the weak$^\ast$ topology in $Y$, and $\vel^k \to \vel$ w.r.t.~the strong topology in $W$ for $K \ni k \to \infty$. Due to \ref{assumptions:weakastconv}--\ref{assumption:4}, $f(\vel) \leq \lim_{K \ni k \to \infty} f(\vel^k) =\ \bar f$. Hence $\vel$ is a solution of \eqref{general:opt:prob::}.
\end{proof}

\subsection{Discretization}

Note that we have not assumed continuity and differentiability of $\psi_\gamma: X \cap Y \to \mathbb R$. It is discussed in \cref{remark::discont} that it does not hold true for the considered application. In order to justify the reformulation \eqref{general:opt:prob::}, we motivate that restriction to suitable discretized subspaces provides a strategy to ensure Fr\'echet differentiability of the regularization term, see \cref{subsec::differentiability}.
Therefore, we investigate a sequence of semidiscretized problems
\begin{equation}
    \min_{\vel^k \in X_k} j(\vel^k) + \psi_\gamma(\vel^k) + \eta(\vel^k),
    \label{general:discretized:opt:prob}
\end{equation}
to approximate the problem
\eqref{general:opt:prob::},
where, for $k \in \mathbb N$, $X_k$ denotes a subspace of $X \cap Z$.
Moreover, we assume the following:
\begin{properties}[{approximation by discretization}]~
    \begin{enumerate}[label=(\textsc{p}\arabic*)]
        \setcounter{enumi}{9}
        \item $Z$ is a
            Banach space such that $Z \hookrightarrow Y$. \label{assumption:1.3}
        \item $X_k \subset X \cap Z$
            are subspaces such that $X_{k} \subset X_{k+1}$ for all $k \in \mathbb N$, and for all $w \in X \cap Z$ there exists a sequence $(w^k)_{k \in \mathbb N}$ with $w^k\in X_k$ such that $ w^k \to w$ as $k \to \infty$ w.r.t. the strong topology in $X$ and  $(\| w^k\|_Z)_{k \in \mathbb{N}}$ is bounded.
            \label{assumption::discretization}
        \item For each bounded sequence $(\vel^k)_{k \in \mathbb N} \subset Z$ and $\vel \in Z$ such that $\vel^k \to \vel$ as $k \to \infty$ w.r.t. the strong topology in $W$, it follows that $\lim_{k \to \infty} \psi_\gamma(\vel^k) = \psi_\gamma(\vel)$. \label{assumption.continuity}
        \item The function $\eta: X \to \mathbb R$ is strongly continuous. \label{assumption:etacont}
    \end{enumerate}
\end{properties}

\begin{remark}
    In general, it is challenging to approximate problems in $Z$ by a sequence of discretized problems if the function space is non-separable (which is the case for our motivating image registration example, for which $Z = W^{1,\infty}_0(\Omega)^d$). However, our relaxation of the norm via $\psi_\gamma$, $\gamma > 0$ allows to circumvent this issue due to \ref{assumption.continuity}, which uses convergence in $W$, in combination with considering the limit $\gamma \to 0$.
\end{remark}

\begin{lemma}
    Let the properties \ref{assumption:1}--\ref{assumption:etacont} be fulfilled. Further, let $\gamma >0$, and let $\vel^k \in X_k$ be a global minimizer of the optimization problem \eqref{general:discretized:opt:prob}. Then there exists a subsequence $(\vel^k)_{k \in K}$, $K \subset \mathbb N$, that converges to a $\vel_\gamma \in X \cap Y$ with respect to the strong topology in $W$, weak topology in $X$ and weak$^\ast$ topology in $Y$, and $\vel_\gamma$ fulfills
    \begin{equation*}
        j(\vel_\gamma) + \psi_\gamma(\vel_\gamma) + \eta(\vel_\gamma) \leq j(w) + \psi_\gamma(w) + \eta(w)
        \quad \text{for all} \quad w \in X \cap Z.
    \end{equation*}
    \label{lemma::discretized_glob}
\end{lemma}

\begin{proof}
    Let $w \in X \cap Z$ be arbitrary but fixed. Due to \ref{assumption::discretization}, there exists a sequence $(w^k)_{k \in \mathbb N}$ with $w^k \in X_k \subset X \cap Z$ for all $k \in \mathbb N$ such that $\lim_{k \to \infty} \| w^k - w\|_{X} = 0$ and  $( \| w^k\|_Z)_{k \in \mathbb N}$ is bounded. \ref{assumption:1} ensures $w^k \to w$ in $W$.
    Thus, due to \ref{assumption.continuity}, \ref{assumption:etacont}, \ref{assumption:1.3} and \ref{assumption:4},
    \begin{equation}
        \lim_{k \to \infty} f(w^k) = f(w), \label{eq::limit_full}
    \end{equation}
    where $f: X \cap Y \to \mathbb R$ is defined as $f(\cdot)\coloneqq j(\cdot)+\psi_\gamma(\cdot) + \eta(\cdot)$.
    By definition of $(\vel^k)_{k \in \mathbb N}$, we further know
    \begin{equation}
        f(\vel^k) \leq  f(w^k) \label{eq::limit_full2}
    \end{equation}
    and, therefore, due to \eqref{eq::limit_full}, $(f(\vel^k))_{k \in \mathbb N}$ is bounded.
    Hence, due to \ref{assumption:14}, there exists a subsequence $(\vel^k)_{k \in K}$, $K \subset \mathbb N$ and $\hat \vel \in X$, $\bar \vel \in Y$, $\vel \in W$ such that $\vel^k \rightharpoonup \hat \vel$ in $X$, $\vel^k \rightharpoonup^\ast \bar \vel$ in $Y$, and, due to \ref{assumption:1}, $\vel^k \to \vel$ in $W$, for $K \ni k \to \infty$. Due to \ref{limit.assumption}, we know that $v = \hat v = \bar v \eqqcolon v_\gamma$ in $X \cap Y$.
    \ref{assumptions:weakastconv}, \ref{assumptions:weakconv} and \ref{assumption:4} yield, combined with \eqref{eq::limit_full2} and \eqref{eq::limit_full}, that
    \begin{equation*}
        f(\vel_\gamma) \leq \liminf_{K \ni k \to \infty} f(\vel^k) \leq \liminf_{K \ni k \to \infty} f(w^k) \leq f(w).
        \qedhere
    \end{equation*}
\end{proof}

\subsection{Limit considerations}
In order to justify the approximated optimization problems \eqref{general:opt:prob::}, we consider the limit
$\gamma \to 0^+$
and the optimization problem \begin{equation}
    \min_{\vel \in X\cap Z} j(\vel) + \beta \| \vel\|_Z + \eta(\vel),
    \label{limiting:optprob}
\end{equation}
for $\beta > 0$.

In order to establish consistency results, we further assume that the following properties hold true.
\begin{properties}[limit considerations]~
    \begin{enumerate}[label=(\textsc{p}\arabic*)]
        \setcounter{enumi}{13}
        \item For each $v \in Y$, $(0,\infty) \ni \gamma \mapsto \psi_\gamma(v) \in \mathbb{R}$ is monotonically decreasing. \label{assumption:estgamma}
        \item Let $\beta>0$, $\psi_0: Y \to \mathbb R$ be such that $\psi_0 (\vel) = \beta \| \vel \|_Z$ for all $\vel \in Z$ and $\psi_0(\vel) = \infty$ for all $\vel \in Y\setminus Z$. It holds that $\lim_{\gamma \to 0^+} \psi_\gamma(\vel) = \psi_0(\vel)$ for arbitrary $\vel \in Y$.  \label{assumption:limgamma}
    \end{enumerate}
\end{properties}

\begin{lemma}
    Let properties \ref{assumption:1}--\ref{assumption:limgamma} be fulfilled.
    Let $(\gamma_\ell)_{\ell \in \mathbb N} \subset (0,\infty)$ be monotonically decreasing sequence that converges to $0$. For $\ell \in \mathbb N$, let $\vel^\ell \in X \cap Y$ be such that
    \begin{equation}
        j(\vel^\ell) + \psi_{\gamma_\ell}(\vel^\ell) + \eta(\vel^\ell) \leq j (w) + \psi_{\gamma_\ell}(w) + \eta(w)
        \label{optimality::eq}
    \end{equation}
    for all $w \in X \cap Z$ (which exist due to \cref{theorem::full,lemma::discretized_glob}). Then the sequence $(\vel^k)_{k \in \mathbb N}$ has a subsequence $(\vel^k)_{k \in K}$, $K \subset \mathbb N$, converging to $\bar \vel$ with respect to the strong topology in $W$, weak topology in $X$ and weak$^\ast$ topology in $Y$, and $\bar v$ solves \eqref{limiting:optprob}.
    \label{theorem::existenceglobsol}
\end{lemma}

\begin{proof}
    Let $\gamma > 0$ and $w \in X \cap Z$ be arbitrary but fixed. For $\gamma_\ell \leq \gamma$, we obtain, due to assumption \ref{assumption:estgamma} and \eqref{optimality::eq},
    \begin{equation}
        \begin{aligned}[t]
            j(\vel^\ell) + \psi_{\gamma}(\vel^\ell) + \eta(\vel^\ell) &\leq j(\vel^\ell) + \psi_{\gamma_\ell}(\vel^\ell) + \eta(\vel^\ell) \\& \leq j(w) + \psi_{\gamma_\ell} (w) + \eta(w) \leq j(w) + \psi_0(w) + \eta(w).
        \end{aligned}
        \label{estimation::limcon}
    \end{equation}
    Hence, $(\psi_{\gamma}(\vel^\ell) + \eta(\vel^\ell))_{\ell \in \mathbb N}$ is bounded, and, as in the proof of \cref{lemma::discretized_glob}, we obtain a subsequence $(\vel^\ell)_{\ell \in L}$, $L \subset \mathbb N$, that converges to a $\bar \vel$ with respect to the strong topology in $W$, weak topology in $X$ and weak$^\ast$ topology in $Y$, and, using \eqref{estimation::limcon}, the limit fulfills
    \begin{equation*}
        j(\bar \vel) + \psi_{\gamma}(\bar \vel) + \eta (\bar \vel) \leq j(w) + \psi_0(w) + \eta(w).
    \end{equation*}
    Since $\gamma$ is arbitrarily chosen the inequality holds for all $\gamma > 0$. Considering the limit $\gamma \to 0^+$ and using assumption \ref{assumption:limgamma} yields
    \begin{equation*}
        j(\bar \vel) + \psi_{0}(\bar \vel) + \eta (\bar  \vel) \leq j(w) + \psi_0(w) + \eta(w).
    \end{equation*}
    Since $w \in X \cap Z$ and due to assumption \ref{assumption:13}, $\psi_0 (\bar \vel) < \infty$ and hence $\bar \vel \in Z$ due to assumption \ref{assumption:limgamma}.
\end{proof}

Moreover, we show that we can approximate a solution of \eqref{limiting:optprob} by a sequence of semidiscretized optimization problems.

\begin{theorem}
    Let properties \ref{assumption:1}--\ref{assumption:limgamma} be fulfilled.
    For $\gamma > 0$, let $v_\gamma^k$ be  a global minimizer of \eqref{general:discretized:opt:prob}.
    Then there exists a sequence $(\gamma_k)_{k \in \mathbb N}$ with $\gamma_k >0 $ for all $k \in \mathbb N$ and $\lim_{k \to \infty} \gamma_k = 0$ such that a subsequence of $(v_{\gamma_k}^k)_{k \in \mathbb N}$ converges (w.r.t.~the strong topology in $W$, the weak topology in $X$, and the weak$^\ast$-topology in $Y$) to $\bar v$, where $\bar v$ is a global minimizer of \eqref{limiting:optprob}. Furthermore, any (w.r.t.~the strong topology in $W$, the weak topology in $X$, and the weak$^\ast$-topology in $Y$) convergent subsequence of $(v_{\gamma_k}^k)_{k \in \mathbb N}$ converges to a global minimizer of \eqref{limiting:optprob}.
    \label{theorem::convergence}
\end{theorem}

\begin{proof}
    Let $\hat v \in X \cap Z$ be a global solution of \eqref{limiting:optprob} (which exists due to \cref{theorem::existenceglobsol}). Let $m \in \mathbb N$ be arbitrary but fixed. Due to \ref{assumption:limgamma}, there exists $\gamma_m > 0$ such that
    \begin{equation}
        | (j(\hat v) + \psi_{\gamma_m}(\hat v) + \eta( \hat v )) - (j(\hat v) + \psi_{0}(\hat v) + \eta( \hat v )) | \leq \frac1{2m}.
        \label{estim1}
    \end{equation}
    Due to \ref{assumption::discretization}, there exists a sequence $(v^k)_{k\in\mathbb N}$ such that $v^k \in X_k$, $\lim_{k\to \infty} \| v^k  - \hat v \|_{X} = 0$, and $(\| v^k\|_{Z})_{k\in \mathbb N}$ is bounded.
    Due to \ref{assumption.continuity}, \ref{assumption:etacont} and \ref{assumption:4}, there exists $k_m\in \mathbb N$ such that
    \begin{equation}
        | (j(\hat v) + \psi_{\gamma_m}(\hat v) + \eta(\hat v)) - (j(v^{k_m}) + \psi_{\gamma_m}(v^{k_m}) + \eta( v^{k_m} )) | \leq \frac{1}{2m}.
        \label{estim2}
    \end{equation}
    We choose $\gamma_m$ and $k_m$ successively for $m \in \mathbb N$ such that $(\gamma_m)_{m\in \mathbb N}$ is monotonically decreasing and $(k_m)_{m\in \mathbb N}$ is monotonically increasing. Combining \eqref{estim1} and \eqref{estim2} yields
    \begin{equation*}
        j(v^{k_m}) + \psi_{\gamma_m}(v^{k_m}) + \eta( v^{k_m})  \leq j(\hat v) +  \beta \| \hat v\|_{Z} + \eta( \hat v ) + \frac1m.
        \label{estim30}
    \end{equation*}
    Due to \ref{assumption:estgamma}, for $\gamma > 0$, we obtain
    \begin{equation}
        j(v^{k_m}) + \psi_{\gamma}(v^{k_m}) + \eta( v^{k_m}) \leq j(\hat v) + \beta \| \hat v\|_{Z} + \eta( \hat v ) + \frac1m.
        \label{estim3}
    \end{equation}
    for all $m \in \mathbb N$ that fulfill $\gamma \geq \gamma_m$.
    Due to the boundedness of $(j(v^{k_m}) + \psi_{\gamma_m}(v^{k_m}) + \eta(v^{k_m}))_{m \in \mathbb N}$, \ref{assumption:1}, \ref{assumption:1.1} and \ref{assumption:14}, there exists a subsequence $(v^{k_m})_{m \in M}$, $M \subset \mathbb N$ that converges to a limit $\bar v \in X \cap Y$ w.r.t. the strong topology in $W$, weak topology in $X$, and weak$^\ast$ topology in $Y$. Due to \ref{assumptions:weakastconv},  \ref{assumptions:weakconv} and \ref{assumption:4}, we have
    \begin{equation}
        j(\bar v) + \psi_{\gamma}(\bar v) + \eta( \bar v ) \leq \liminf_{M \ni m \to \infty} j(v^{k_m}) + \psi_{\gamma}(v^{k_m}) + \eta( v^{k_m}).
        \label{estim4}
    \end{equation}
    Combining \eqref{estim3} and \eqref{estim4} yields
    \begin{equation*}
        j(\bar v) + \psi_{\gamma}(\bar v) + \eta( \bar v ) \leq j(\hat v) + \beta\| \hat v\|_{Z} + \eta( \hat v ).
    \end{equation*}
    Since the estimate holds for all $\gamma >0$, due to \ref{assumption:limgamma},
    \begin{equation*}
        j(\bar v) + \beta \| \bar v\|_{Z} + \eta(\bar v) \leq j(\hat v) + \beta\| \hat v\|_{Z} + \eta(\hat v ),
    \end{equation*}
    which implies the assertion.
\end{proof}

\section{Orlicz spaces}
\label{sec::preliminaries}

We consider Orlicz spaces as defined in \cite{kufner1977function}. Let $\Phi$ be a (relaxed) Young function according to \cite[Def.~4.2.1, and Rem.~4.2.9]{kufner1977function}, i.e., there exists a function $\varphi: [0, \infty) \to [0,\infty)$ such that $\Phi(t) \coloneqq \int_0^t \varphi(s) \mathrm d s$, $ t \geq 0$ with $\varphi(0) = 0$, $\varphi(s) \geq 0$, $\varphi(s)$ is left-continuous at any $s > 0$, $\varphi$ is nondecreasing on $[0, \infty)$, and $\lim_{s\to \infty} \varphi(s) = \infty$.
Moreover, let $\Omega \subset \mathbb R^d$ be a bounded open domain with measure $\mu$ such that $\int_\Omega\mathrm{d}\mu = 1$. The Orlicz space $\mathcal L^\Phi(\Omega)$ is defined by
\begin{equation}
    \mathcal L^\Phi (\Omega) \coloneqq \mathcal L^\Phi(\Omega,  \mu) \coloneqq \left\lbrace f: \Omega \to \mathbb R \text{ measurable } : ~ \exists \eta > 0 \text{ s.t. } \int_\Omega \Phi(\eta|f(x)|) \mathrm d \mu(x) < \infty \right\rbrace.
\end{equation}

\begin{remark}
    In the scope of this work, for bounded domains $\Omega \subset \mathbb R^d$
    that satisfy $|\Omega| > 0$, where $|\Omega|$ denotes the Lebesgue measure of $\Omega$, we work with $\mathrm d \mu \coloneqq \frac{1}{|\Omega|}\mathrm{d} x$.
    \label{remark::mudef}
\end{remark}

For a measurable function $f$ on $\Omega$, its Luxemburg norm is defined by

\begin{equation}
    \| f \|_\Phi \coloneqq \inf \left\lbrace \gamma > 0 \,:\, \int_\Omega \Phi\left(\frac{|f(x)|}{\gamma}\right) \mathrm d \mu (x) \leq 1 \right\rbrace.
\end{equation}

\begin{proposition}[{cf. \cite[Prop.~2.2]{simon2011convexity} and \cite[Lem.~2.17]{Turett1980}}]
    Let $\Phi$ be a Young function, and $\Omega \subset \mathbb R^d$ be such that $\int_\Omega \mathrm d \mu = 1$. Then $(\mathcal L^\Phi(\Omega), \|\cdot\|_{\Phi})$ is a Banach space. Moreover, there holds that
    \begin{equation*}
        \int_\Omega \Phi\left(\frac{|f(x)|}{\| f\|_{\Phi}}\right) \mathrm{d} \mu(x) \leq 1
        \label{eq::estimate}
        \quad \text{for all} \quad f \in \mathcal L^\Phi(\Omega), \quad  f \geq 0.
    \end{equation*}
\end{proposition}

\begin{proposition}[{cf. \cite[Lem.~2.6]{clason2021entropic}}]
    If $\Phi$ is a Young function, and $\Omega \subset \mathbb R^d$ satisfies $\int_\Omega \mathrm d \mu = 1$, then
    \begin{equation}
        \| f \|_\Phi \leq \max\left( \int_\Omega \Phi(|f(x)|) \mathrm{d} \mu(x), 1\right)
        \quad \text{for all} \quad f \in \mathcal L^\Phi(\Omega).
    \end{equation}
    \label{lem::maxnorm}
\end{proposition}

We work with the space $L_{\exp}(\Omega) \coloneqq \mathcal L^{\Phi_{\exp}}(\Omega)$,
where
\begin{equation}
    \Phi_{\exp}(t) \coloneqq \begin{cases}
        t \quad &\text{if } 0 \leq t < 1, \\
        \eu^{t-1} \quad & \text{if } t \geq 1.
    \end{cases}
\end{equation}
The Fenchel--Legendre conjugate of a convex function $\Phi: I \to \mathbb R$ with $I \subset \mathbb R$ is defined by $\Phi^\ast(t) \coloneqq \sup_{s \in I} s t - \Phi(s)$. Then, by definition of $\Phi^\ast$, the following lemma holds.
\begin{proposition}[Fenchel--Young inequality]
    Let $I \subset \mathbb R$, and $\Phi: I\to \mathbb R$ be a convex function. Then,
    \begin{equation*}
        yz \leq  \Phi(y) + \Phi^\ast(z) \quad \text{for all} \quad y,z \in I.
    \end{equation*}
    \label{lem::estimate_sum}
\end{proposition}

\begin{remark}
    The function $\Phi_{\log}$ defined by
    \begin{equation}
        \Phi_{\log} (t) \coloneqq \begin{cases} t \ln^+ t \quad &\text{if } t > 0, \\ 0 & \text{if } t = 0,  \end{cases}
    \end{equation}
    with $\ln^+ (t) \coloneqq \max(0, \ln(t))$, is the Fenchel--Legendre transform of $\Phi_{\exp}$, see \cite[Example 4.3.6]{kufner1977function}. The function $\Phi_{\log}$
    is a (relaxed) Young function \cite[Rem. 4.2.10, and Example 4.1.3 (ii)]{kufner1977function}.
    \label{psilog}
\end{remark}

\begin{proposition}[{cf. \cite[Thm.~3.1]{Turett1980}}]
    Let $\Omega \subset \mathbb R^d$ be such that $\int_\Omega \mathrm d \mu = 1$.
    Further, let $\Phi^\ast$ be the Fenchel--Legendre transform of a convex function $\Phi: I \to \mathbb R$, $I \subset \mathbb R$. If
    $f \in \mathcal L^\Phi (\Omega)$  and $g \in \mathcal L^{\Phi^\ast} (\Omega)$,
    then
    \begin{equation}
        \int_\Omega f(x) g(x) \mathrm{d}\mu (x) \leq  2\| f\|_{\Phi} \| g \|_{\Phi^\ast}.
    \end{equation}
    \label{lem::hoelder_generalization}
\end{proposition}

\begin{proposition}[cf. {\cite[Sec.~4.4, Thm.~4.11.1]{kufner1977function}}]
    If $\Phi_{\log}$ is given as in \cref{psilog}, then the space $L{\log}L (\Omega) = \mathcal L^{\Phi_{\log}}(\Omega) $ is separable.
    \label{lemma::5}
\end{proposition}

\begin{remark}
    The space $L_{\exp}(\Omega)$ is not separable \cite[Example 4.4.5, Remark 4.11.2]{kufner1977function}, and can be identified with the dual space of $L{\log}L(\Omega)$ \cite[Thm.~6.5]{Bennett1988}. Moreover, the spaces $L_{\exp}(\Omega)$ and $L{\log}L(\Omega)$ are not reflexive \cite[Thm.~7.7.1]{Kosmol2011}.
    \label{predualLexp}
\end{remark}

\section{Approximation to essential supremum}
\label{sec::approxnorm}

Before discussing results for an approximation of a norm on $W^{1,\infty}_0(\Omega)^d$ (\cref{subsec::w1infty}), we consider properties of an approximation to the essential supremum $\mathrm{ess\,sup}_\Omega$ on $\Omega$ and its dual representation (\cref{subsec::dualrepresentation}). We discuss convexity, monotonicity and asymptotic properties (\cref{subsec::cma}, \ref{assumption:estgamma}, \ref{assumption:limgamma}), continuity and Fr\'echet differentiability (\cref{subsec::differentiability}, \ref{assumption.continuity}), and prove sequential weak-$\ast$ lower semicontinuity (\cref{subsec::semicontinuity}, \ref{assumptions:weakastconv}, \ref{assumption:14}).
We consider the mapping 
\begin{equation*}
E_\gamma:\quad L_{\exp}(\Omega) \to \mathbb R, \quad u \mapsto \gamma \ln\left(\frac{1}{|\Omega|} \int_\Omega \exp(\gamma^{-1} u(x))\mathrm d x\right).
\end{equation*}

\subsection{Dual Representation}
\label{subsec::dualrepresentation}

Our analysis makes use on the following dual representation of the essential supremum and its approximation.
\begin{lemma}
    Let $\Omega \subset \mathbb R^d$, $d \geq 1$, be a bounded domain and $u \in L_{\exp}(\Omega)$. Moreover, let
    \begin{equation*}
        \Pi \coloneqq \left\lbrace \pi \in L{\log}L(\Omega): \ B(\pi)< \infty, \ \frac{1}{|\Omega|}\int_\Omega \pi(x) \mathrm d x =1 \right\rbrace,
    \end{equation*}
    where 
    $B: L\log L(\Omega) \to \mathbb R \cup \lbrace + \infty \rbrace$  defined by
    \begin{equation*}
        B(\pi) \coloneqq \begin{cases} \frac{1}{|\Omega|} \int_{\Omega} \pi(x) \ln(\pi(x)) \mathrm{d}x \quad &\text{if } \pi(x) > 0 \text{ a.e.}, \\
            + \infty  & \text{otherwise,}
        \end{cases}
    \end{equation*}
    is a barrier for the constraint $\pi \geq 0$.
    Then, for $\gamma > 0$,
    \begin{align}
        E_\gamma(u) &= \sup_{\pi \in \Pi} \frac{1}{|\Omega|} \int_\Omega u(x) \pi(x) \mathrm d x - \gamma B(\pi),
        \label{dualrepresentatonofEgamma} \\ \shortintertext{and}
        \mathrm{ess\,sup}_\Omega(u) &= \sup_{\pi \in \Pi} \frac{1}{|\Omega|} \int_\Omega u(x) \pi(x) \mathrm d x.
        \label{esssup::eq}
    \end{align}
    \label{lem::asymptotic}
\end{lemma}

The proof of the lemma is built on \cite[Prop. 2.2, Prop. 2.3, C.3, pp. 527ff]{budhiraja2019analysis}. Before giving a proof, we set the statement into relation to results from the literature.

\begin{remark}
    \label{DonskerVaradhan}
    We first demonstrate that the definition of $B(\pi)$ is closely related to the concept of relative entropy introduced in \cite[Chap. 2]{budhiraja2019analysis}. Let $(\mathcal V, \mathcal A)$ be a measurable space and $\mathcal P(\mathcal V)$ be the set of all probability measures on $(\mathcal V, \mathcal A)$. For $\theta \in \mathcal P(\mathcal V)$, the relative entropy $R(\cdot \| \theta): \mathcal P(\mathcal V) \to \overline{\mathbb R}$ is defined via
    \begin{equation*}
        R(\rho \| \theta) = \begin{cases}
        \int_{\mathcal V} \log( \frac{d \rho}{d \theta}) d \rho \quad &\text{if } \rho \ll \theta, \\
        + \infty & \text{else},
        \end{cases}
    \end{equation*}
    where $\rho \ll \theta$ denotes that $\rho$ is absolutely continuous w.r.t. $\theta$.
    In our setting, $\mathcal V = \Omega$, $\mathcal A = \mathcal B(\Omega)$ denotes the Borel algebra on $\Omega$, and $\theta = \frac{1}{|\Omega|} \mathcal L \in \mathcal P(\mathcal V)$ is a probability measure, where $\mathcal L$ denotes the Lebesgue measure on $\Omega$. For $\pi \in L{\log}L(\Omega)$ such that $\frac{1}{|\Omega |} \int_{\Omega} \pi(x) \mathrm d x = 1$ and $\pi \geq 0$ a.e., $\rho$ defined via $\mathrm d \rho = \pi \, \mathrm d\theta$ is a probability measure. From this, it follows that (choosing the natural logarithm)
    \begin{equation*}
        R(\rho \| \theta) = \frac{1}{|\Omega |} \int_{\Omega} \pi(x) \ln(\pi(x)) \mathrm d x = B(\pi)
    \end{equation*}
    for all $\pi \in L{\log}L(\Omega)$ such that $\frac{1}{|\Omega|} \int_{\Omega} \pi(x) \mathrm d x = 1$ and $\pi \geq 0$ a.e..
    The Donsker--Varadhan variational formula \cite[Lem. 2.4(a)]{budhiraja2019analysis} yields that, for a Polish space $\mathcal X$, for each $\rho, \theta \in \mathcal P(\mathcal X)$, 
    \begin{equation*}
        R(\rho \| \theta) = \sup_{\psi \in \mathcal M_b(\mathcal X)} \left( \int_{\mathcal X} \psi \mathrm d \rho - \log\left(\int_{\mathcal X} \mathrm{exp}(\psi) \mathrm d \theta \right)\right),
    \end{equation*}
    where $\mathcal M_b(\mathcal X)$ denotes the space of bounded Borel-measurable functions from $\mathcal X$ to $\mathbb R$.
    Translating to our setting, $\mathcal X = \Omega$ is a Polish space as an open subset of a Polish space ($\mathbb R^{d}$ equipped with the Euclidean distance). 
\end{remark}

Before presenting the proof of \cref{lem::asymptotic}, we first present a direct consequence of the Donsker--Varadhan variational formula, cf. \cref{DonskerVaradhan}. For the sake of completeness, we present a direct proof of this assertion following \cite[Lem. 1.4.1 and Prop. 1.4.2]{dupuisellis}.

\begin{lemma}
    Let $\Omega \subset \mathbb R^d$, $d \geq 1$, be a bounded Lipschitz domain, $\pi \in \Pi$ and $\vela \in L^\infty(\Omega)$. Then
    it holds that 
    \begin{equation*}
        B(\pi) \geq \frac{1}{|\Omega|}\int_\Omega \vela(x) \pi(x) \mathrm d x - \ln\left(\frac{1}{|\Omega|} \int_\Omega \exp(\vela(x)) \mathrm d x\right).
    \end{equation*}
    \label{lem::donskvar_estim}
\end{lemma}

\begin{proof}
    We define 
    \begin{equation*}
        \pi_0(x) \coloneqq \frac{\exp(u(x))}{\frac1{|\Omega|} \int_\Omega \exp(u(x)) \mathrm d x}.
    \end{equation*}
    It holds that
    \begin{equation}
        B(\pi) = \frac{1}{|\Omega|} \int_\Omega \pi(x) \ln(\pi(x)) \mathrm{d} x = \frac{1}{|\Omega|}\int_\Omega \ln (\pi_0(x)) \pi(x)\mathrm dx + \frac{1}{|\Omega|} \int_\Omega \ln\left(\frac{\pi(x)}{\pi_0(x)} \right) \pi(x) \mathrm d x.
        \label{eq::reform1}
    \end{equation}
    We observe that, using the definition of $\pi_0$ and the fact that $\frac{1}{|\Omega|} \int_\Omega \pi(x) \mathrm dx = 1$, the first summand of \eqref{eq::reform1} can be written as
    \begin{equation*}
        \frac{1}{|\Omega|}\int_\Omega \ln (\pi_0(x)) \pi(x)\mathrm dx = \frac{1}{|\Omega|}\int_\Omega u(x) \pi(x)\mathrm dx - \ln\left(\frac{1}{|\Omega|} \int_\Omega \exp(u(x)) \mathrm d x\right).
    \end{equation*}
    For the second summand of \eqref{eq::reform1}, we use $s \ln(s) \geq s -1 $ for all $s > 0$ to obtain the estimate
    \begin{equation*}
        \frac{1}{|\Omega|}\int_\Omega \ln\left(\frac{\pi(x)}{\pi_0(x)}\right) \pi(x) \mathrm d x \geq \frac{1}{|\Omega|}\int_\Omega \left(\frac{\pi(x)}{\pi_0(x)} - 1 \right) \pi_0(x) \mathrm d x = 0,
    \end{equation*}
    since $\frac{1}{|\Omega|}\int_\Omega \pi(x) \mathrm d x = 1 = \frac{1}{|\Omega|}\int_\Omega \pi_0(x) \mathrm d x$. This completes the proof.
\end{proof}

\begin{proof}[Proof of \cref{lem::asymptotic}]
    Let $u \in L_{\exp}(\Omega)$ and $\gamma > 0$. In case of $\frac{1}{|\Omega|}\int_\Omega \exp(\gamma^{-1} u(x)) \mathrm d x = \infty$, for $N \in \mathbb N$, we consider $f_N: \Omega \to \mathbb R$ defined by $f_N(x) \coloneqq \min(N, \frac1\gamma u(x))$. Since $f_N$ is bounded from above, 
    \begin{equation*}
        \pi_N \coloneqq \frac{1}{Z_N} \exp(f_N), \quad Z_N \coloneqq \frac{1}{|\Omega|} \int_\Omega \exp(f_N(x)) \mathrm d x,
    \end{equation*}
    fulfills $B(\pi_N) < \infty$, and $\frac{1}{|\Omega|} \int_\Omega \pi_N(x)\mathrm d x = 1$. By Fatou's lemma and the assumption on $u$, we further know that $\lim_{N \to \infty}Z_N = \infty$. Hence, using that $f_N \leq \gamma^{-1} u(x)$, 
    \begin{equation*}
        \begin{aligned}
            \sup_{\pi \in \Pi} \left(\frac{1}{|\Omega|} \int_\Omega u(x) \pi(x) \mathrm d x - \gamma B(\pi) \right) 
                & \geq \frac{1}{|\Omega|} \int_\Omega u(x) \pi_N(x) \mathrm d x - \gamma B(\pi_N) \\
                & = \frac{\gamma}{|\Omega|} \int_\Omega (\gamma^{-1} u(x) - f_N(x)) \frac{\exp(f_N(x))}{Z_N} \mathrm d x + \gamma \ln(Z_N) \geq \gamma \ln(Z_N),
        \end{aligned}
    \end{equation*}
    which diverges to $\infty$ as $N \to \infty$. If $\frac{1}{|\Omega|} \int_\Omega \exp(\gamma^{-1} u(x)) \mathrm d x < \infty$, taking the limit and using Fatou's lemma implies that 
    \begin{equation*}
        \sup_{\pi \in \Pi} \left(\frac{1}{|\Omega|} \int_\Omega u(x) \pi(x) \mathrm d x - \gamma B(\pi)\right) \geq \gamma \ln\left( \frac{1}{|\Omega|} \int_\Omega \exp\left(\gamma^{-1} u(x)\right) \mathrm d x \right) = E_\gamma(u).
    \end{equation*}
    In order to show equality, it remains to prove that 
    \begin{equation}
        B(\pi) \geq \frac{1}{|\Omega|} \int_\Omega \gamma^{-1} u(x) \pi(x) \mathrm d x - \ln\left( \frac{1}{|\Omega|} \int_\Omega \exp\left( \gamma^{-1} u(x) \right) \mathrm d x \right)
        \label{eq::toprove}
    \end{equation}
    for all $\pi \in \Pi$. For arbitrary but fixed $M \in \mathbb N$, let $F_M: \Omega \to [-M, M]$ be defined by \begin{equation}
        F_M(x) \coloneqq \min(\max(\gamma^{-1} u(x), -M), M).
        \label{eq::definitionFM}
    \end{equation} Due to \cref{lem::donskvar_estim}, 
    \begin{equation*}
        B(\pi) \geq \frac{1}{|\Omega|} \int_\Omega F_M(x) \pi(x) \mathrm d x - \ln\left(\frac{1}{|\Omega|} \int_\Omega \exp\left(F_M(x)\right) \mathrm d x\right).
    \end{equation*}
    Let w.l.o.g. $\pi \in \Pi$ be such that $\int_\Omega F_M(x) \pi(x) \mathrm d x > -\infty$ (otherwise the statement holds directly). Then, using the dominated convergence theorem for the integrals
    \begin{equation*}
        \int_{\lbrace \gamma^{-1} u < 0 \rbrace} \exp\left( F_M(x) \right) \mathrm d x \quad\text{and}\quad \int_\Omega F_M(x) \pi(x) \mathrm dx,
    \end{equation*}
    and the monotone convergence theorem for the integral
    \begin{equation*}
        \int_{\lbrace \gamma^{-1} u \geq 0 \rbrace}\exp\left( F_M(x) \right) \mathrm d x
    \end{equation*}
    yields \eqref{eq::toprove}.
    Hence, we have shown that \eqref{dualrepresentatonofEgamma} holds. 
    
    In order to show \eqref{esssup::eq}, we observe that, for all $\pi \in \Pi$ and for $u \in L_{\exp}(\Omega)$ such that $\mathrm{ess\,sup}_\Omega(u) < \infty$, it holds that 
    \begin{equation}
        \frac{1}{|\Omega|} \int_\Omega u(x) \pi(x) \mathrm dx \leq \mathrm{ess\,sup}_\Omega (u).
        \label{eq::esssupestim}
    \end{equation}
    On the other hand, for every $ n \in \mathbb N$, there exists a measurable set $D_n \subset \Omega$ of positive measure such that $u(x) \geq \mathrm{ess\,sup}_\Omega (u) - \frac1n$ for all $x \in D_n$. The function $\pi_n \coloneqq \frac{1}{|D_n|} 1_{D_n}$ is an element of $\Pi$ and it holds that 
    \begin{equation*}
        \frac{1}{|\Omega|} \int_\Omega u(x) \pi_n (x) \mathrm d x \geq \mathrm{ess\, sup}_\Omega(u) - \frac{1}{n}
    \end{equation*}
    for all $n \in \mathbb N$. This, together with \eqref{eq::esssupestim}, implies \eqref{esssup::eq}, whenever $\mathrm{ess\,sup}_\Omega(u) < \infty$. If $\mathrm{ess\,sup}_\Omega(u) > \infty$, $D_n$ is chosen as a measurable set of positive measure such that $u(x) \geq n$ for all $x \in D_n$, which implies that $\sup_{\pi \in \Pi} \frac{1}{|\Omega|} \int_\Omega u(x) \pi(x) \mathrm d x \geq n$ for all $n \in \mathbb N$. This concludes the proof.
\end{proof}

\subsection{Convexity, monotonicity, and asymptotic behavior}
\label{subsec::cma}

We next collect some elementary properties of the approximation.
\begin{lemma}
    The mapping $E_\gamma: L_{\exp}(\Omega) \to \mathbb R\cup \lbrace \infty \rbrace$ is convex.
\end{lemma}

\begin{proof}
    This follows from the convexity of $$\vela \mapsto \sup_{\substack{\pi \in L\log{}L(\Omega) :~B(\pi) < \infty,\\ \frac{1}{|\Omega|} \int_{\Omega} \pi(x) \mathrm{d}x = 1}} \frac{1}{|\Omega|} \int_{\Omega } \vela(x) \pi(\x) \mathrm{d}x - \gamma B(\pi)$$ due to the convexity of the pointwise supremum of convex functions, and using the fact that the first summand is linear in $u$ and $\gamma B(\pi)$ does not depend on $u$.
\end{proof}

\begin{lemma}
    Let $u \in L_{\exp}(\Omega)$ and $\gamma_0 \geq \gamma > 0$, then we obtain $E_\gamma (\vela) \geq E_{\gamma_0}(\vela)$.
    \label{lem::gamma}
\end{lemma}
\begin{proof}
    It holds with Jensen's inequality (the mapping $[0,\infty) \to \mathbb R$, $x \mapsto x^\frac{\gamma}{\gamma_0}$ is concave) and due to the monotonicity of the logarithm that
    \begin{equation*}
        \begin{aligned}
            \gamma \ln\left(\frac{1}{|\Omega|}\int_{\Omega} \exp(\gamma^{-1} u(x)) \mathrm{d}x \right)
            & \geq \ln\left(\left(\frac{1}{|\Omega|}\int_{\Omega } \exp({\gamma^{-1} u(x)})^\frac{\gamma}{\gamma_0} \mathrm{d}x \right)^{\gamma_0}\right)
            \\
            & = \gamma_0 \ln\left(\frac{1}{|\Omega|}\int_{\Omega} \exp(\gamma_0^{-1} u(x)) \mathrm{d}x \right),
        \end{aligned}
    \end{equation*}
    from which the claim follows.
\end{proof}

\begin{remark}
    For fixed $\gamma > 0$, the space of functions for which $\ln( \frac1{|\Omega|} \int_\Omega \exp(\gamma^{-1} \vela(x)) \mathrm d x )$ is bounded is larger than
    $L^\infty(\Omega)$. Consider, e.g., $\vela: (0,1) \to \mathbb R, x \mapsto \ln(x^{-\lambda})$ for $\lambda > 0$ such that $\frac{\lambda}\gamma <1$.
    We have $\vela \notin L^\infty((0,1))$. However, $\int_0^1 \exp(\frac1\gamma \vela(x)) \mathrm d x = \int_0^1 \exp(\frac1\gamma \ln(x^{-\lambda})) \mathrm d x = \int_0^1 x^\frac{-\lambda}\gamma \mathrm d x = \frac{\gamma}{\gamma - \lambda}$. Therefore, $\gamma \ln( \frac1{|\Omega|} \int_\Omega \exp(\frac1\gamma \vela(x)) \mathrm d x )$ is bounded.
\end{remark}

We can now show the asymptotic behavior of the approximation.
\begin{lemma}
    Let $\Omega \subset \mathbb R^d$, $d \geq 1$, be a bounded domain.
    Then for all $\vela \in L_{\exp}(\Omega)$, we have
    \begin{equation}
        E_\gamma(\vela)
        \to \mathrm{ess~sup}_{\Omega}(\vela)
        \quad \text{as} \quad \gamma \to 0^{+},
        \quad
        \text{and}
        \quad
        E_\gamma(\vela)
        \to \frac{1}{|\Omega|} \int_{\Omega} \vela(x) \, \mathrm{d} x
        \quad \text{as} \quad \gamma \to \infty,
        \label{eq::limit}
    \end{equation}
    and
    for all $\gamma > 0$,
    \begin{equation*}
        \frac{1}{|\Omega|} \int_{\Omega} \vela(x) \, \mathrm{d} x\leq E_\gamma(\vela)
        \leq \mathrm{ess~sup}_{\Omega}(\vela).
    \end{equation*}
    \label{estimates::lem}
\end{lemma}

\begin{proof}
    Due to \Cref{lem::gamma} it suffices to show that \eqref{eq::limit} holds. Analogously to \cite[Example 6.20]{Shapiro2021}, the limit $\gamma \to \infty$ is obtained by applying L'H\^opital's rule to the mapping 
    \begin{equation*}
        \gamma \mapsto \frac{\ln\left(\frac{1}{|\Omega|}\int_{\Omega} \exp(\gamma^{-1} u(x)) \mathrm{d}x \right)}{\frac1\gamma} = E_\gamma(\vela).
    \end{equation*}
    For the limit $\gamma \to 0^+$, we consider the representation of $E_\gamma$ given in \eqref{dualrepresentatonofEgamma}.
    Let $u \in L_{\exp}(\Omega)$ be such that $\mathrm{ess\,sup}_\Omega (u) < \infty$.
    On the one hand, we know that, for $u \in L_{\exp}(\Omega)$ such that $\mathrm{ess\,sup}_\Omega(u) < \infty$, 
    \begin{equation*}
        E_\gamma(u) \leq \mathrm{ess\,sup}_\Omega (u) + \gamma\, \mathrm{sup}_{\pi \in \Pi}  (-B(\pi)).
    \end{equation*}
    Since $\sup_{\pi \in \Pi} (-B(\pi))$ is bounded from above, taking the limit $\gamma \to 0^+$ yields
    \begin{equation}
        \limsup_{\gamma \to 0^+} E_\gamma(u) \leq \mathrm{ess\,sup}_\Omega (u).
        \label{eq::upperbound}
    \end{equation}
    Let $(\pi_n)_{n \in \mathbb N} \subset \Pi$ be such that $\lim_{n \to \infty} \frac{1}{|\Omega|} \int_{\Omega} u(x) \pi_n(x)\mathrm d x = \mathrm{ess\,sup}_\Omega(u)$. Then, since $B(\pi_n) < \infty$ for all $n \in \mathbb N$, we know that there exists a monotonically increasing sequence $(N_k)_{k \in \mathbb N} \in \mathbb N$ and a monotonically decreasing sequence $(\gamma_k)_{k\in \mathbb N} \subset (0,\infty)$ such that 
    \begin{equation*}
        E_{\gamma_k} (u) \geq \mathrm{ess\,sup}_\Omega(u) - \frac1k.
    \end{equation*}
    Taking the limit $k \to \infty$ and using monotonicity of $\gamma \to E_\gamma(u)$ (see \cref{lem::gamma}) yields 
    \begin{equation*}
        \liminf_{\gamma \to 0^+} E_{\gamma}(u) \geq \mathrm{ess\,sup}_\Omega(u)
    \end{equation*}
    and, with \eqref{eq::upperbound},
    \begin{equation*}
        \lim_{\gamma \to 0^+} E_{\gamma}(u) = \mathrm{ess\,sup}_\Omega(u)
    \end{equation*}
    With similar arguments it can also be shown that $\lim_{\gamma \to 0^+} E_{\gamma}(u) = \infty$ if $\mathrm{ess\,sup}_\Omega(u) = \infty$.
\end{proof}

\begin{remark}
    Since $\Omega$ is bounded, $\mu$ defined by
    $\mathrm{d}\mu \coloneqq (1/|\Omega|) \mathrm{d}x$ (cf. \cref{remark::mudef})
    is a probability measure on $\Omega$.
    Hence $\vela \mapsto E(\vela)$
    defined on the space of measurable mappings on $\Omega$
    equals the entropic risk measure
    $\vela \mapsto  \gamma \ln \mathbb{E}_{\mu}[\exp( \gamma^{-1} \vela)]$
    defined on the space of measurable mappings on $\Omega$ equipped with $\mu$,
    where $\mathbb{E}_\mu$ is the expectation with respect to $\mu$.
    The results are therefore related to the results derived in \cite[Example 6.20]{Shapiro2021}.
\end{remark}

\subsection{Continuity and Fr\'echet differentiability}
\label{subsec::differentiability}

The next lemma shows that the approximation is continuous and -- unlike the essential supremum -- even Fréchet differentiable in $L^\infty$.
\begin{lemma}
    The mapping 
    \begin{equation*}
    E_\gamma: ~L^\infty(\Omega) \to \mathbb R, \quad  \vela \mapsto \gamma \ln\left( \frac{1}{|\Omega|}\int_\Omega \exp(\gamma^{-1} \vela(x)) \mathrm{d}x\right)
    \end{equation*}
    is continuous and Fr\'echet differentiable. Moreover, for $r \in [1,\infty]$ and
    for any sequence $(\vela_k)_{k \in \mathbb N} \subset L^\infty(\Omega)$ with bounded $(\|\vela_k\|_{L^\infty(\Omega)})_{k \in \mathbb N}$ and $\lim_{k\to \infty} \| \vela_k - \vela \|_{L^r(\Omega)} = 0$ for $\vela \in L^\infty(\Omega)$ it follows that $$\lim_{k \to \infty} E_\gamma(\vela_k) = E_\gamma(\vela).$$
    \label{remark::continuity}
\end{lemma}
\begin{proof}

    We observe that $E_\gamma = F_1 \circ F_2 \circ F_3$ with
    \begin{equation*}
        \begin{aligned}
        & F_1: (0,\infty) \to \mathbb R, \qquad\qquad &&z \mapsto \gamma \ln(z), \\
        & F_2: L^1(\Omega) \to \mathbb R, \qquad \qquad &&f \mapsto \frac{1}{|\Omega|}\int_\Omega f(x) \mathrm d x, \\
        & F_3: L^\infty(\Omega) \to L^\infty(\Omega), \quad\ &&\vela \mapsto \exp({\gamma^{-1} \vela(x)}).
        \end{aligned}
    \end{equation*}
    Due to \cite[Lem.~4.11 and Lem.~4.12]{troltzsch2010optimal}, the Nemytskii operator $F_3$ is continuous and Fr\'echet differentiable. Moreover, for any $M \in \mathbb R$ and $r \in [1,\infty]$, and all $\vela_1, \vela_2 \in L^\infty(\Omega)$ with $\|\vela_1\| \leq M$, $\|\vela_2 \|\leq M$, we have
    \begin{equation*}
        \| F_3(\vela_1) - F_3(\vela_2) \|_{L^r(\Omega)} \leq L_M \| \vela_1 - \vela_2\|_{L^r(E)},
        \label{continuity::nemytskii}
    \end{equation*}
    where $L_M$ denotes a constant depending on the choice of $M$. Since $F_3(\vela)(x) > 0$ for all $x\in \Omega$ and $\vela \in L^\infty(\Omega)$, $F_2 \circ F_3 (\vela) > 0$, and the composition is well-defined. Moreover, $F_1$ and $F_2$ are continuous and Fr\'echet differentiable.
\end{proof}

\begin{remark}
    Given $\gamma > 0$, the function $E_\gamma: L_{\exp}(\Omega) \to \mathbb R\cup \lbrace \infty \rbrace$ is not continuous.
    Let $\vela \in L_{\exp}(\Omega)$ for which $\int_\Omega \exp(\gamma^{-1} \vela(x))\mathrm d x < \infty$, then there exists in general no open $L_{\exp}(\Omega)$ neighborhood $U$ around $u$ such that $\int_\Omega \exp(\gamma^{-1} \tilde \vela(x)) \mathrm{d}x < \infty$ for all $\tilde \vela \in U$. For example, for $\Omega = (0, \frac12)$, $\vela(x)\coloneqq \gamma \ln( (\ln(x)^2 x)^{-1})$ and $h_\alpha(x) \coloneqq \alpha \vela(x)$ it holds that $ u(x) + h_\alpha(x) = \gamma \ln ( (\ln(x)^2 x)^{-1-\alpha}) $. Hence, $\int_\Omega \exp({\gamma^{-1} (\vela(x) + h_\alpha(x))}) \mathrm{d}x = \infty$ for any $\alpha > 0$.\label{remark::discont}
\end{remark}

\subsection{Sequential \texorpdfstring{weak$^\ast$}{weak-star} lower semicontinuity}
\label{subsec::semicontinuity}

We now consider the smooth approximation of an $L^\infty$-seminorm, which we recall from \eqref{formula_chigamma} is given by
\begin{equation*}
    \chi_\gamma(\vela) \coloneqq \gamma \ln\left(\frac{1}{|\Omega|}\int_{\Omega} \exp({\gamma^{-1} \vela(x)}) \mathrm{d} x\right) + \gamma \ln\left(\frac{1}{|\Omega|}\int_\Omega \exp({-\gamma^{-1} \vela(x)}) \mathrm{d}x\right),
\end{equation*}
We first show a coercivity property.
\begin{lemma}
    Let $\Omega \subset \mathbb R^d$, $d \geq 1$, be a bounded domain, and $\gamma >0$.
    Let $(\vela_n)_{n \in \mathbb N} \subset L^2(\Omega)$ be such that $\left(\chi_\gamma(\vela_n) + \frac12 \| \vela_n\|_{L^2(\Omega)}^2\right)_{n \in \mathbb N}$ is bounded. Then $\left(\| \vela_n \|_{L_{\exp}(\Omega)}\right)_{n \in \mathbb N}$ is bounded.
    \label{lem::seqweak_mod}
\end{lemma}

\begin{proof}
    Due to the properties of the norm, it holds
    \begin{equation}
        \| \vela_n \|_{L_{\exp}(\Omega)} = \gamma \| \gamma^{-1} \vela_n \|_{L_{\exp}(\Omega)}.
        \label{eq::proof::1}
    \end{equation}
    Hence, due to \cref{lem::maxnorm} it suffices to show that $\frac{1}{|\Omega|}\int_\Omega \Phi_{\exp}(|\frac1\gamma \vela_n|) \mathrm d x$ is bounded.
    Due to \cref{estimates::lem}, we know that
    \begin{equation}
        \begin{multlined}[t]
            \chi_\gamma(\vela_n) = \gamma \ln\left(\frac{1}{|\Omega|}\int_\Omega \exp({\gamma^{-1} \vela_n(x)}) \mathrm d x \right) + \gamma \ln\left(\frac{1}{|\Omega|}\int_\Omega \exp({-\gamma^{-1} \vela_n(x)}) \mathrm d x \right)\\ \geq \frac{1}{|\Omega|} \int_\Omega u_n(x)\mathrm d x + \frac{1}{|\Omega|} \int_\Omega - u_n(x)\mathrm d x = 0.
        \end{multlined}
    \label{eq::proof::2}
    \end{equation}
    Therefore, boundedness of $\left(\chi_\gamma (\vela_n) + \frac12 \| \vela_n\|_{L^2(\Omega)}^2\right)_{n \in \mathbb N}$ yields a constant $C > 0$ such that
    \begin{equation}
        \frac12 \| \vela_n\|_{L^2(\Omega)}^2 \leq C \text{ and } \chi_\gamma (\vela_n) \leq C
        \label{L2estim}
    \end{equation}
    for all $n \in \mathbb N$.

    Next, we observe that, due to the definition of $\Phi_{\exp}$,
    \begin{equation}
        \frac{1}{|\Omega|}\int_\Omega \Phi_{\exp}(|\gamma^{-1} \vela_n|) \mathrm d x \leq \frac{1}{|\Omega|}\int_\Omega |\gamma^{-1} \vela_n| \mathrm{d} x + \frac{1}{|\Omega|}\int_\Omega \exp({|\gamma^{-1} \vela_n(x)| - 1}) \mathrm d x
        \label{eq::proof::3}
    \end{equation}
    By \eqref{L2estim} and Hölder's inequality, we know that the first term is bounded:
    \begin{equation}
        \frac{1}{|\Omega|}\int_\Omega |\gamma^{-1} \vela_n (x)| \mathrm{d} x \leq \frac{1}{\gamma |\Omega|^{1/2}} \| \vela_n\|_{L^2(\Omega)} \leq \frac{\sqrt{C}}{\sqrt{2} \gamma |\Omega|^{1/2}}.
        \label{eq::proof::4}
    \end{equation}
    Moreover, since we know that
    \begin{equation}
        \begin{aligned}[t]
            \frac{1}{|\Omega|}\int_\Omega \exp({|\gamma^{-1} \vela_n(x)|-1}) \mathrm d x &\leq \frac{1}{\eu |\Omega|} \int_\Omega \exp({\gamma^{-1} \vela_n(x)}) + \exp({-\gamma^{-1} \vela_n(x)}) \mathrm d x \\
            &= \frac{1}{\eu|\Omega|} \int_\Omega \exp({\gamma^{-1} \vela_n(x)}) \mathrm{d} x + \frac{1}{\eu|\Omega|} \int_\Omega \exp({-\gamma^{-1} \vela_n(x)}) \mathrm d x,
        \end{aligned}
        \label{eq::proof::5}
    \end{equation}
    where $\eu = \mathrm{exp}(1)$, it suffices to bound the latter two terms.
    Starting with \eqref{formula_chigamma}, we obtain with \cref{estimates::lem}
    \begin{equation}
        \begin{aligned}[t]
            \gamma \ln\left(\frac{1}{|\Omega|} \int_\Omega \exp({\gamma^{-1} \vela_n(x)}) \mathrm{d} x\right)  &\leq \chi_\gamma(\vela_n) - \frac12 \|\vela_n\|_{L^2(\Omega)}^2 - \gamma \ln\left(\frac{1}{|\Omega|} \int_\Omega \exp({- \gamma^{-1} \vela_n(x)}) \mathrm{d} x\right) \\
            &\leq \chi_\gamma(\vela_n) - \frac12 \|\vela_n\|_{L^2(\Omega)}^2 + \frac{1}{|\Omega|}\int_\Omega  |\vela_n(x)| \mathrm{d} x ,
        \end{aligned}
        \label{eq::proof::6}
    \end{equation}
    and, analogously,
    \begin{equation}
        \gamma \ln\left(\frac{1}{|\Omega|}\int_\Omega \exp(-\gamma^{-1} \vela_n(x)) \mathrm{d} x\right) \leq \chi_\gamma(\vela_n) - \frac12 \|\vela_n\|_{L^2(\Omega)}^2 + \frac{1}{|\Omega|}\int_\Omega  |\vela_n(x)| \mathrm{d} x.
        \label{eq::proof::7}
    \end{equation}
    Therefore, combining \eqref{eq::proof::5}, \eqref{eq::proof::6} and \eqref{eq::proof::7}, and due to \eqref{L2estim} and \eqref{eq::proof::4},
    \begin{equation}
        \begin{aligned}[t]
            \frac{1}{|\Omega|}\int_\Omega \exp({|\gamma^{-1} \vela_n(x)|-1}) \mathrm d x
            &\leq \frac{2}{\eu } \exp\left(\gamma^{-1} \left(\chi_\gamma(\vela_n) - \frac12 \|\vela_n\|_{L^2(\Omega)}^2 + \frac{1}{|\Omega|}\int_\Omega  |\vela_n(x)| \mathrm{d} x \right)\right) \\
            &\leq \frac{2}{\eu} \exp\left(\gamma^{-1} \left(C + \frac{\sqrt{C}}{\sqrt{2} |\Omega|^{1/2}} \right)\right) =: \tilde C.
        \end{aligned}
        \label{eq::proof::8}
    \end{equation}
    Combining \eqref{eq::proof::1}, \cref{lem::maxnorm}, \eqref{eq::proof::3} and \eqref{eq::proof::8} yields
    \begin{equation*}
        \| \vela_n \|_{L_{\exp}(\Omega)} \leq \gamma \left( 1 + \frac{1}{|\Omega|} \int_\Omega \Phi_{\exp} \left(|\gamma^{-1} \vela_n(x)|\right) \mathrm d x  \right) \leq \gamma\left( 1 +  \frac{\sqrt{C}}{\sqrt{2}\gamma|\Omega|^{1/2}} +  \tilde C\right)
    \end{equation*}
    for all $n \in \mathbb N$.
\end{proof}

Finally, we show weak$^*$ lower semicontinuity of $\chi$ in $L_{\exp}$.
\begin{lemma}
    Let $(\vela_n)_{n \in \mathbb N} \subset L_{\exp}(\Omega)$ be such that there exists a constant $C> 0$ with $\| \vela_n \|_{L_{\exp}(\Omega)} \leq C$, $\gamma > 0$ and $\chi_\gamma$ be defined by \eqref{formula_chigamma}. Then there exists a subsequence $(\vela_n)_{n \in K}$, $K \subset \mathbb N$, and $\vela \in L_{\exp}(\Omega)$ such that $\vela_n \rightharpoonup^\ast \vela$ in $L_{\exp}(\Omega)$ and $\chi_\gamma(\vela) \leq \liminf_{K \ni n \to \infty} \chi_\gamma(\vela_n)$.
    \label{lem::lsc_exp}
\end{lemma}
\begin{proof}
    The proof is similar to \cite[Proof of Thm.~2.3]{AcarVogel}.
    Due to boundedness of $(\vela_n)_{n \in \mathbb N}$ in $L_{\exp}(\Omega)$ and Banach--Alaoglu's theorem and the separability of the pre-dual $L{\log}L(\Omega)$ of $L_{\exp}(\Omega)$, we know that $(\vela_n)_{n \in \mathbb N}$ contains a subsequence $(\vela_n)_{n \in K}$, $ K \subset \mathbb N$, that converges $L_{\exp}(\Omega)$-weakly$^\ast$ to $\vela \in L_{\exp}(\Omega)$.
    We exemplarily show the argumentation for the first summand of $\chi_\gamma$.    
    Let $\pi \in \Pi$
    be arbitrarily chosen.
    Then we know that
    \begin{equation*}
        \begin{aligned}
            \frac{1}{|\Omega|}\int_{\Omega } \vela(x)  \pi(x) \mathrm{d}x - \gamma B(\pi) &= \lim_{K \ni n \to \infty} \frac{1}{|\Omega|}\int_{\Omega }\vela_n(x) \pi(x) \mathrm{d}x - \gamma B(\pi) \\
                 &= \liminf_{K \ni n \to \infty} \frac{1}{|\Omega|}\int_{\Omega }\vela_n(x)) \pi(x) \mathrm{d}x - \gamma B(\pi) \\
                 & \leq \liminf_{K \ni n \to \infty} \sup_{\pi \in \Pi}\frac{1}{|\Omega|}\int_{\Omega }\vela_n(x)) \pi(x) \mathrm{d}x - \gamma B(\pi)
                 = \liminf_{K \ni n \to \infty} E_\gamma(\vela_n).
        \end{aligned}
    \end{equation*}
    Taking the supremum over all $\pi \in \Pi$ and doing a similar argumentation for the second summand yields the statement.
\end{proof}

\subsection{Approximation of \texorpdfstring{$\scriptstyle W^{1,\infty}_0(\Omega)^d$}{W1infty}-norm}
\label{subsec::w1infty}

We consider $W^{1,\infty}_0(\Omega)^d \coloneqq \lbrace v \in W^{1,\infty}(\Omega)^d~:~ \mathrm{tr}(v) = 0 \rbrace$, where $\mathrm{tr}$ denotes the trace, as a closed linear subspace of $W^{1,\infty}(\Omega)^d$ equipped with the norm
\begin{equation}
        \|v \|_{W^{1,\infty}_0(\Omega)^d} \coloneqq \sum_{j=1}^d \sum_{k=0}^1 \left(\mathrm{ess~sup}_{\Omega}((-1)^k v_j ) + \sum_{i=1}^d \mathrm{ess~sup}_{\Omega}((-1)^k \partial_i v_j)\right).
    \label{w1inftynorm}
\end{equation}
Note that working with a definition of $W^{1,\infty}_0(\Omega)^d$ requires caution since different non-equivalent notions are used, see \cite[Remark 10.10]{Leoni}. 

We approximate \eqref{w1inftynorm} via
\begin{equation}
    \begin{multlined}[t]
        \Psi_{\gamma}(v) \coloneqq   \sum_{j =1}^d  \sum_{k=0}^1  \Bigg( \gamma \ln\Big( \frac{1}{|\Omega|}\int_{\Omega } \exp\left(\gamma^{-1}{(-1)^k} v_j(x) \right)\mathrm{d} x\Big) \\ +  \sum_{i=1}^d \gamma \ln\Big(\frac{1}{|\Omega|} \int_{\Omega } \exp\left({\gamma^{-1}{(-1)^k}  \partial_i v_j(x) }\right)\mathrm{d} x\Big)\Bigg).
    \end{multlined}
    \label{psigamma}
\end{equation}
Due to \cref{estimates::lem}, $\lim_{\gamma \to 0} \Psi_{\gamma}(v)$ is equal to $\| v \|_{W^{1,\infty}_0(\Omega)^d}$. Moreover, the following property follows directly from \cref{lem::gamma}.

\begin{corollary}
    Let $v \in W_0^{1, \exp}(\Omega)^d$ defined by
    \begin{equation*}
        W_0^{1, {\exp}}(\Omega)^d \coloneqq \left\lbrace v \in W^{1,\exp}(\Omega)^d ~:~ \mathrm{tr}(v) = 0 \right\rbrace,
    \end{equation*}
    where $W^{1,\exp}(\Omega)^d = W^1 L_{\exp}(\Omega)^d$
    denotes the Orlicz-Sobolev space which contains all function in $L_{\exp}(\Omega)$ whose distributional derivatives are in $L_{\exp}(\Omega)$ (see \cite{donaldsontrudinger}),
    $\gamma_0 \geq \gamma > 0$ and $\Psi_\gamma$ be defined via \eqref{psigamma}. Then we obtain
    \begin{equation*}
        \Psi_{\gamma} (v) \geq \Psi_{\gamma_0}(v).
    \end{equation*}
    \label{corollary::estimpsi}
\end{corollary}

\section{Results for the transport equation}
\label{sec::transporteq}

By providing existence (\cref{subsec::existence}) and uniqueness of solutions (\cref{subsec::uniqueness}), this section proves well-definedness of the reduced objective function $$j(v)\coloneqq J(\phi(\cdot, T), \phi_{\mathrm{tar}}),$$ where $\phi$ solves the linear hyperbolic transport equation, cf.~\eqref{opt::prob}. Moreover, \cref{subsec::stability} establishes a stability result (cf. \ref{assumption:4}).

\subsection{Existence of solutions}
\label{subsec::existence}

Let $\Omega \subset \mathbb{R}^d$ with $d \geq 1$ be a bounded domain with Lipschitz boundary, and let $T> 0$.
We consider the transport equation
\begin{equation}
    \begin{aligned}
        \partial_t \phi + \veli \cdot \nabla \phi &= 0 &&\text{in } (0,T) \times \Omega, \\
        \phi(0,\cdot) &= \phi_0 && \text{on } \Omega,
    \end{aligned}
    \label{eq::transport}
\end{equation}
where $\veli$ denotes the velocity field and $\phi_0$ is the initial condition.

\begin{definition}[{cf. \cite[Def.~3.1.1]{jarde2018analysis}}]
    Let $\veli \in L^1((0,T) \times \Omega)^d$ with distributional divergence $\mathrm{div}(\veli) \in L^1((0,T)\times \Omega)$, and let $\phi_0 \in L^\infty(\Omega)$. Then we call a function
    \begin{equation*}
        \phi \in C([0,T], L^\infty(\Omega)\mathrm{-w}^\ast)
    \end{equation*}
    a \emph{weak solution} of \eqref{eq::transport} if
    \begin{equation}
        \int_0^T \int_\Omega \phi(t,x) (\varphi_t(t,x) + \veli(t,x) \cdot \nabla \varphi(t,x) + \varphi(t,x) \mathrm{div}(\veli(t,x))) \mathrm{d} x \, \mathrm \mathrm{d}t= - \int_\Omega \phi_0 (x) \varphi(0,x) \mathrm d x
        \label{eq::weak_form}
    \end{equation}
    for all $\varphi \in C_c^\infty ([0,T)\times\Omega)$.
    \label{definition}
\end{definition}

\begin{remark}[cf. {\cite[Sec.~2.1.3]{jarde2018analysis}}]
    For a separable Banach space $X$, $L^p((0,T), X)$ denotes the space of Bochner integrable functions with values in $X$.
    For $X$ being the non-separable dual space of a separable Banach space $P$, we denote by $L^p((0,T), X)$ the space of Gelfand integrable functions with values in $X$, i.e., the function that are weak$^\ast$ measurable ($f: (0,T) \to \langle f(t), p \rangle_{X, P}$ is Lebesgue measurable for any $p \in P$, which also implies that $t \to \| f(t)\|_X$ is Lebesgue measurable) and fulfill
    $\int_0^T \| f(t)\|_X^p\, \mathrm d t < \infty$.
    Moreover, $L^\infty((0,T), X)$ denotes the space of weak$^\ast$ measurable functions $f: (0,T) \to X$ such that $t \mapsto \| f(t)\|_X$ is essentially bounded in $(0,T)$. A reason why one works with weak$^\ast$ measurability is due to the fact that functions that describe moving fronts, like, e.g., $u: [0,1] \to L^\infty((0,1))$, $u(t)(x) = 1_{\lbrace x \geq t\rbrace }(t,x)$, where $1_{\lbrace x \geq t \rbrace}(t,x)$ denotes the indicator function of the set $\lbrace (t,x) \in [0,1]\times [0,1]\,:\, x \geq t \rbrace$, are not Bochner measurable.
\end{remark}
\begin{remark}[{cf. \cite{crippa2008flow}}]
    Let $u:~(0,T) \times \Omega \to \mathbb R$ such that $u(t, \cdot) \in L^\infty(\Omega)$ for all $t \in [0,T]$ and $\varphi \in L^1(\Omega)$. Define $\Phi_{\varphi}(u): [0,T] \to \mathbb R$ via $$\Phi_{\varphi}(u)(t):= \int_\Omega u(t, x) \varphi(x) \mathrm{d}x.$$
    By definition,
    \begin{equation*}
        C([0,T], L^\infty(\Omega)\mathrm{-w}^\ast) = \lbrace u \in L^\infty((0,T), L^\infty(\Omega))\,:\, \Phi_{\varphi}(u) \in C([0,T]) \text{ for all } \varphi \in L^1(\Omega) \rbrace.
    \end{equation*}
    \label{remark::linftyweak}
\end{remark}

\begin{proposition}
    $(C([0,T], L^\infty(\Omega)\mathrm{-w}^\ast), \| \cdot \|_{L^\infty((0,T), L^\infty(\Omega))})$ is a Banach space, where $$ \| \cdot \|_{L^\infty((0,T), L^\infty(\Omega))} := \mathrm{ess\,sup}_{t \in (0,T)} \| \cdot(t) \|_{L^\infty(\Omega)}.$$
    \label{proposition::Bspace}
\end{proposition}
\begin{proof}
    In order to prove this statement, we show completeness.
    Let $$(u_n)_{n \in \mathbb N} \subset C([0,T], L^\infty(\Omega)\mathrm{-w}^\ast)$$ be a Cauchy sequence. Since $L^\infty((0,T), L^\infty(\Omega))$ is a Banach space, there exists $\tilde u \in L^\infty((0,T), L^\infty(\Omega))$ such that $\lim_{n \to \infty} \| u_n - \tilde u \|_{L^\infty((0,T), L^\infty(\Omega))} = 0$. Let $\tilde u$ be a representative of the equivalence class of functions in $L^\infty((0,T), L^\infty(\Omega))$. Then we know that
    \begin{equation}
        \| \tilde u(t, \cdot) - u_n(t, \cdot) \|_{L^\infty(\Omega)} \to 0 \quad \text{for } t \in I_T,
        \label{completeness1}
    \end{equation}
    where $I_T \subset [0,T]$ and $[0,T]\setminus I_T$ is a set with Lebesgue measure $0$. Let $u: [0,T] \times \Omega \to \mathbb R$ be such that $u(t, \cdot) = \tilde u(t, \cdot)$ for all $t \in I_T$. Moreover, for $s \in [0,T]\setminus I_T$, there exists $(s_n)_{n \in \mathbb N} \subset I_T$ such that $\lim_{n \to \infty} s_n = s$. Since $(\|u(s_n, \cdot)\|_{L^\infty(\Omega)})_{n \in \mathbb N}$ is bounded, there exists a weakly$^\ast$ convergent subsequence with limit $u_s \in L^\infty(\Omega)$. We set $u(s, \cdot):= u_s(\cdot)$ for all $s \in [0,T]\setminus I_T$. Since we only changed $\tilde u$ on a set of measure zero, $\| u_n - u\|_{L^\infty((0,T), L^\infty(\Omega))} = 0$ and $u(t, \cdot) \in L^\infty(\Omega)$ for all $t\in [0,T]$. What remains to show for proving $u \in C([0,T], L^\infty(\Omega)\mathrm{-w}^\ast)$ is that $\Phi_\varphi(u) \in C([0,T])$ for all $\varphi \in L^1(\Omega)$. Let $\varphi \in L^1(\Omega)$ be arbitrary but fixed, and $\epsilon > 0$. By definition of $u$, for every $t, s \in [0,T]$ and $\delta > 0$, there exist $t_{\delta,\epsilon}, s_{\delta,\epsilon} \in I_T$ such that $|t_{\delta, \epsilon} - t| \leq \frac\delta4$, $|s_{\delta, \epsilon} - s| \leq \frac\delta4$, and
    \begin{equation}
        |\Phi_\varphi(u)(t) - \Phi_\varphi(u)(s)| \leq |\Phi_\varphi(u)(t_{\delta, \epsilon}) - \Phi_\varphi(u)(s_{\delta, \epsilon})| + \frac{\epsilon}4
        \label{completeness2}
    \end{equation}
    Due to \eqref{completeness1} and the fact that $u(t, \cdot) = \tilde u(t,\cdot)$ for all $t \in I_T$, there exists $N > 0$ such that
    \begin{equation}
        |\Phi_\varphi(u)(\tau) - \Phi_\varphi(u_n)(\tau)| \leq \frac{\epsilon}4
        \label{completeness3}
    \end{equation}
    for all $\tau \in I_T$ and $n \geq N$.
    Now, let $t \in [0,T]$ be arbitrary but fixed. Due to continuity of $\Phi_\varphi(u_N)$, there exists $\delta > 0$ such that
    \begin{equation}
        |\Phi_\varphi(u_N)(t) - \Phi_\varphi(u_N)(s)| \leq \frac{\epsilon}4
        \label{completeness4}
    \end{equation}
    for all $s \in B_\delta(t)$. Combination of \eqref{completeness2}, \eqref{completeness3}, \eqref{completeness4} yields that for every $\epsilon > 0$ there exists $\delta > 0$ such that
    \begin{equation*}
        |\Phi_\varphi(u)(t) - \Phi_\varphi(u)(s)| \leq \epsilon
    \end{equation*}
    for all $s \in B_{\delta/2}(t)$, which ensures that $|s_{\delta,\epsilon}- t_{\delta,\epsilon}|\leq \delta$. This implies $\Phi_\varphi(u)\in C([0,T])$.
\end{proof}

Due to results from \cite{jarde2019existence, crippa2008flow, crippa2014initial}, the following existence result holds (see \cref{sec::proofexistence}).

\begin{proposition}[existence, cf. {\cite[Thm. 3.1.4, Rem. 3.1.6]{jarde2018analysis}}]
    Let $\Omega \subset \mathbb R^d$ be a bounded domain with Lipschitz boundary,
    $\phi_0 \in L^\infty(\Omega)$, and $\veli \in L^1((0,T), \mathrm{BV}_0(\Omega))^d$ satisfy $\mathrm{div}(\veli) \in L^1((0,T) \times \Omega)$. Then there exists a weak solution $\phi \in C([0,T], L^\infty(\Omega) \mathrm{-w}^\ast)$ of \eqref{eq::transport}, which fulfills
    \begin{equation*}
        \|\phi (t, \cdot)\|_{L^\infty(\Omega)} \leq \| \phi_0 \|_{L^\infty(\Omega)}
        \quad \text{for all} \quad t \in [0,T].
    \end{equation*}
    \label{lem::existence}
\end{proposition}

\subsection{Uniqueness of solutions}
\label{subsec::uniqueness}

In order to prove uniqueness of solutions in \cref{theorem::uniqueness}, we use the so-called renormalization property.
\begin{definition}[{cf. \cite[Def.~3.1.7 and Def.~3.1.8]{jarde2018analysis}}]
    Let $\veli$ be a vector field in $L^1((0,T) \times \Omega)^d$ with $\mathrm{div}(\veli) \in L^1((0,T) \times \Omega)$. Then $\veli$ has the \emph{renormalization property} if for every $\phi_0 \in L^\infty(\Omega)$, every solution $\phi \in L^\infty((0,T) \times \Omega)$ of the transport equation \eqref{eq::transport}, with vector field $\veli$ and initial condition $\phi_0$, is a renormalized solution, i.e., for each $C^1$-function $\beta: \mathbb R \to \mathbb R$, the function $\beta(\phi)$ is a weak solution of
    \begin{equation}
        \begin{aligned}
            \partial_t \beta(\phi) + \veli \cdot \nabla \beta (\phi) &= 0 && \text{in } (0,T) \times \Omega, \\
            \beta(\phi(0,\cdot)) &= \beta (\phi_0) && \text{in } \Omega.
        \end{aligned}
        \label{eq::transport_renormalized}
    \end{equation}
\end{definition}

For smooth $\phi$ and $\veli$ this property follows directly from \cite[Thm.~4]{chen2011image}, see \cite[Sec.~3.1.2]{jarde2018analysis}. From \cite[Rem. 2.2.2]{crippa2008flow}, we obtain $\beta(\phi) \in C([0,T], L^\infty(\Omega) \mathrm{-w}^\ast)$.

\begin{proposition}[{cf. \cite[Thm.~3.1.26]{jarde2018analysis}}]
    Let $\Omega \subset \mathbb R^d$ be a bounded domain with Lipschitz boundary, and
    $\veli$ be a vector field in $L^1((0,T), \mathrm{BV}_0(\Omega))^d \cap L^\infty((0,T)\times \Omega)^d$  such that $\mathrm{div}(\veli) \in L^1((0,T) \times \Omega)$. Then $\veli$ has the renormalization property.
    \label{lem::renormalization}
\end{proposition}

An essential tool for proving uniqueness is the following version of Grönwall's inequality.
\begin{proposition}[Grönwall type inequality, cf. {\cite[p. 356, Thm.~4]{mitrinovic1991inequalities}}]
    Let $w: [0,T] \to [0, \infty)$ be continuous, non-negative and satisfy
    \begin{equation*}
        w(t) \leq c_0 + \int_0^t c_1(s) w(s) \mathrm{d}s,
    \end{equation*}
    where $c_1: [0,T] \to [0, \infty)$, $c_1 \in L^1((0,T))$, is a non-negative integrable function and $c_0 \geq 0$.
    Then
    \begin{equation*}
        w(t) \leq c_0 \mathrm{exp}\left(\int_0^t c_1(s) \mathrm d s\right)
    \end{equation*}
    for all $t \in [0,T]$.
    \label{lem::gronnwall}
\end{proposition}

\begin{theorem}[uniqueness]
    Let $\Omega \subset \mathbb R^d$ be bounded domain with Lipschitz boundary, let $\phi_0 \in L^\infty(\Omega)$, and let $\veli$ be a vector field with $$\veli \in L^1((0,T), \mathrm{BV}_0(\Omega))^d \cap L^\infty((0,T)\times \Omega)^d, \quad \text{and} \quad  \mathrm{div}(\veli) \in L^p((0,T), L_{\exp}(\Omega)),$$ $p\in(1,\infty]$. Then
    there exists a unique weak solution $\phi \in C([0,T], L^\infty(\Omega) \mathrm{-w}^\ast)$ of \eqref{eq::transport}.
    \label{theorem::uniqueness}
\end{theorem}

\begin{proof}
    Let $\phi_1, \phi_2\in C([0,T], L^\infty(\Omega) \mathrm{-w}^\ast)$ be arbitrary bounded solutions of \eqref{eq::transport} with vector field $\veli$ and initial condition $\phi_0$, which exist due to \cref{lem::existence}. Hence, $\phi\coloneqq \phi_1 - \phi_2$ is a solution to \eqref{eq::transport} with vector field $\veli$ and initial condition $0$.
    By \cref{remark::linftyweak} there exists $\tilde C > 0$ such that
    \begin{equation}\| \phi \|_{L^\infty((0,T), L^\infty(\Omega))} \leq \|\phi_1\|_{L^\infty((0,T), L^\infty(\Omega))} + \|\phi_2\|_{L^\infty((0,T), L^\infty(\Omega))} \leq \tilde C.
        \label{eq::tildeCestim}
    \end{equation}
    Moreover, due \cref{lem::renormalization}, $\veli$ has the renormalization property, i.e., $\beta(\phi)$ fulfills \eqref{eq::transport_renormalized} and
    \begin{equation*}
        \beta(\phi) \in C([0,T], L^\infty(\Omega) \mathrm{-w}^\ast)
    \end{equation*}
    for all $\beta \in C^1(\mathbb R)$.

    Since $\mathrm{div}(\veli) \in L^p((0,T), L_{\exp}(\Omega))$ due to Hölder's inequality, we know that
    \begin{equation*}
        \int_r^{\min(r+s, T)} \| \mathrm{div}(\veli(t,\cdot)) \|_{L_{\exp}(\Omega)} \mathrm{d}t \leq s^\frac1q \| \mathrm{div}(b)\|_{L^p((0,T), L_{\exp}(\Omega))},
    \end{equation*}
    with for any $r \in [0,T]$, $s > 0$, $q = 1$ if $p = \infty$, and $\frac1q + \frac1p = 1$ else. Thus, there exists $\bar t \in (0,T]$ (independent of the choice of $r$) such that
    \begin{equation}
        \int_r^{s} \| \mathrm{div}(\veli(t,\cdot)) \|_{L_{\exp}(\Omega)} \mathrm{d}t \leq \frac18,
        \label{eq::deltathat__}
    \end{equation}
    for any $r \in [0,T]$ and all $s \in [r, \min(r + \bar t, T)]$.

    \Cref{lem::renormalization} yields that, for an arbitrary $C^1$-function $\beta: \mathbb R \to \mathbb R$, $\beta(\phi)$ is a solution to \eqref{eq::transport_renormalized} with initial condition $\beta(0)$ and vector field $\veli$. Hence, due to \eqref{eq::weak_form} and with the same argumentation as in \cite[Sec.~3.1.1]{jarde2018analysis} (which allow to replace $C_c([0,T) \times \Omega)$ by $C_c([0,T] \times \mathbb R^d)$ if $\veli \in L^1((0,T), \mathrm{BV}_0(\Omega))^d \cap L^\infty((0,T)\times \Omega)^d$ in \cref{definition}), $\phi$ fulfills
    \begin{equation}
        \int_0^{T} \int_\Omega \beta(\phi(t,x)) (\varphi_t (t, x) + \veli (t,x) \cdot \nabla \varphi (t,x) + \varphi (t,x) \mathrm{div}(\veli(t,x))) \mathrm{d} x \, \mathrm{d}t+ \int_\Omega \beta(\phi(0)) \varphi(0,x) \mathrm{d} x =0
        \label{eq::weakform__}
    \end{equation}
    for arbitrary $\varphi \in C_c^\infty([0,T) \times \mathbb R^d)$.

    Let $\tau \in [0,T]$ be arbitrary but fixed.
    We choose $\varphi (x,t) = \varphi_n(t)$, where $\varphi_n: \mathbb R \to [0,1]$ denotes a sequence of mollifications of the function $\mathbb R \to [0,1]$, $t \mapsto H(\tau - t)$, where $H$ denotes the Heaviside function defined as $H(s) = \begin{cases} 1 \quad & \text{if } s \geq 0, \\ 0 & \text{if } s< 0.\end{cases}$\\ Due to \eqref{eq::weakform__}, we have
    \begin{equation*}
        \int_0^T \int_\Omega \beta(\phi(t,x)) ((\varphi_n)_t (t) + \veli (t,x) \cdot \nabla \varphi_n (t) + \varphi_n (t) \mathrm{div}(\veli(t,x))) \mathrm{d} x \, \mathrm \mathrm{d}t + \int_\Omega \beta(\phi(0)) \varphi_n(0) \mathrm{d} x =0,
    \end{equation*}
    and, hence, for $n \to \infty$, since $t \mapsto \int_\Omega \beta(\phi(t,x)) \mathrm{d} x$ is continuous due to $\beta(\phi) \in C([0,T], L^\infty(\Omega)\mathrm{-w^\ast})$, $\lim_{n \to \infty} \int_0^T (\varphi_n)_t(t) \psi(t) \mathrm{d} t = -\psi(\tau)$ for any $\psi \in C([0,T])$, and $\nabla \varphi_n(t) = 0$,
    \begin{equation}
        \int_\Omega \beta(\phi(\tau, x)) \mathrm{d} x   =  \int_\Omega \beta(\phi(0, x))  \mathrm{d} x + \int_0^{\tau} \int_\Omega \beta(\phi(t,x))  \mathrm{div}(\veli(t,x))\mathrm{d} x \, \mathrm d t.
        \label{eq::integral__}
    \end{equation}

    For arbitrary but fixed $\gamma > 0$, we choose 
    \begin{equation*}
        \beta(\phi) \coloneqq \eu^{\frac1\gamma \phi^2 + 1}. \label{defbeta}
    \end{equation*}
    Let $r \geq 0$ be chosen such that $\phi(t, \cdot) = 0$ a.e. for all $t \in [0,r]$, which is fulfilled for $r = 0$.
    Since $\beta(\phi(0)) = \eu$, using \eqref{eq::integral__} with $\tau = \min(r + \bar t, T)$, and the definition of $r$, we obtain
    \begin{equation}
        \begin{multlined}
            \int_\Omega \beta(\phi(\bar t, x)) \mathrm{d} x   =  |\Omega| \eu + \int_r^{\min(r + \bar t,T)} \int_\Omega \beta(\phi(t,x))  \mathrm{div}(\veli(t,x))\mathrm{d} x \, \mathrm d t.
        \end{multlined}
        \label{proof_unique:eq:1}
    \end{equation}
    We consider the latter integral. Due to \cref{lem::hoelder_generalization},
    we know that
    \begin{equation*}
        \frac1{|\Omega|}\int_\Omega \beta(\phi(t,x))  \mathrm{div}(\veli(t,x)))\mathrm{d} x
        \leq  2  \|  \beta(\phi(t, \cdot)) \|_{L{\log} L(\Omega)} \| \mathrm{div}(\veli(t,\cdot)) \|_{L_{\exp}(\Omega)}.
        \label{proof_unique:eq:2}
    \end{equation*}
    Since $\beta(\phi(t, x))\geq \eu$, we know that $\| \beta (\phi(t, \cdot))\|_{L{\log}L(\Omega)}>1$ and $\frac1{|\Omega|} \int_\Omega \beta(\phi(t,x)) \ln_+ (\beta(\phi(t,x))) \mathrm{d} x > 1$.
    Therefore, due to \cref{lem::maxnorm},
    \begin{equation}
        \begin{aligned}[t]
            \| \beta(\phi(t,\cdot))\|_{L{\log} L(\Omega)}
            &\leq \max\left(\frac{1}{|\Omega|} \int_{\Omega}  \beta(\phi(t,x)) \ln_+ (\beta(\phi(t,x))) \mathrm{d} x, 1\right)\\
            & = \frac{1}{|\Omega|} \int_{\Omega} \beta(\phi(t,x)) \ln (\beta(\phi(t,x))) \mathrm{d} x \\
            &\leq \ln\left(\eu^{1 + \frac1\gamma \hat C^2}\right) \frac1{|\Omega|} \int_\Omega \beta(\phi(t,x))\mathrm dx = \left(1 + \frac1\gamma \hat C^2\right) \frac1{|\Omega|} \int_\Omega \beta(\phi(t,x))\mathrm dx,
        \end{aligned}
        \label{proof_unique:eq:3}
    \end{equation}
    for any $\hat C$ such that $\phi(t,\cdot)^2 \leq \hat C^2$ a.e. in $\Omega$ and for a.e. $t \in [r,\min(r+\bar t, T)]$, which holds for $\hat C = \tilde C$ due to \eqref{eq::tildeCestim}.
    Combining \eqref{proof_unique:eq:1}--\eqref{proof_unique:eq:3} yields
    \begin{equation*}
        \frac{1 }{|\Omega|} \int_{\Omega} \beta(\phi(t,x)) \mathrm{d} x  \leq \eu + \int_r^{\min(r + \bar t, T)} 2\left(1 + \gamma^{-1} \hat C^2\right) \left(\frac{1 }{|\Omega|} \int_{\Omega} \beta(\phi(t,x)) \mathrm{d} x\right)   \| \mathrm{div}(\veli(t,\cdot)) \|_{L_{\exp}(\Omega)}  \mathrm{d} t.
    \end{equation*}
    Application of \cref{lem::gronnwall} with  $w(t) = \frac{1 }{|\Omega|} \int_{\Omega} \beta(\phi(t,x)) \mathrm{d} x$, $c_1(t) = 2(1 + \frac1\gamma \hat C^2)\| \mathrm{div}(\veli(t,\cdot)) \|_{L_{\exp}(\Omega)}$ and $c_0 = \eu$ yields
    \begin{equation*}
        \frac{1}{|\Omega|} \int_\Omega \beta(\phi(t,x))\mathrm d x \leq \mathrm{exp}\left(1+ \int_r^{t} 2\left(1 + \gamma^{-1} \hat C^2\right)\| \mathrm{div}(\veli(\tau,\cdot)) \|_{L_{\exp}(\Omega)} \mathrm d \tau\right)
    \end{equation*}
    for a.e. $t \in [r, \min(r + \bar t, T)]$ and therefore, since $\phi \in C([0,T], L^\infty(\Omega)\mathrm{-w^\ast})$, for all $t \in [0,\bar t]$.
    Taking the logarithm on both sides and multiplying by $\gamma$ yields
    \begin{equation*}
        \gamma \ln\left(\frac{1}{|\Omega|} \int_\Omega \beta(\phi(t,x))\mathrm d x\right) \leq \gamma \left(1+ \int_r^t 2\left(1 + \gamma^{-1} \hat C^2\right)\| \mathrm{div}(\veli(\tau,\cdot)) \|_{L_{\exp}(\Omega)} \mathrm d \tau\right),
    \end{equation*}
    and taking the limit $\gamma \to 0^+$ and using \cref{estimates::lem}, we know that
    \begin{equation*}
        \|\phi(t, \cdot)^2\|_{L^\infty(\Omega)} \leq \int_r^t 2 \hat C^2 \| \mathrm{div}(\veli(\tau,\cdot)) \|_{L_{\exp}(\Omega)} \mathrm d \tau
    \end{equation*}
    for all $t \in [r, \min(r + \bar t, T)]$.
    Hence, with \eqref{eq::deltathat__}, we obtain
    \begin{equation*}
        \|\phi(t,\cdot)^2\|_{L^\infty(\Omega)} \leq \bigg(\frac{\hat C}{2}\bigg)^2,
    \end{equation*}
    and iteratively repeating the argumentation with $\frac{\hat C}2$ instead of $\hat C$ yields
    \begin{equation*}
        \|\phi(t,\cdot)^2\|_{L^\infty(\Omega)} \leq \bigg(\frac{\hat C}{2^k}\bigg)^2,
    \end{equation*}
    for all $t \in [r, \min(r + \bar t, T)]$ and $k \in \mathbb N$. Hence,
    \begin{equation}
        \|\phi(t,\cdot)^2\|_{L^\infty(\Omega)} \leq 0,
        \label{proof::unique::end}
    \end{equation}
    for all $t \in [r, \min(r + \bar t, T)]$.
    This shows that $\phi( t, \cdot) = 0$ a.e. for all $t \in [r,\min(R,T)]$ with $R := r + \bar t$. Hence, as long as $R < T$,
    the argumentation \eqref{proof_unique:eq:1}--\eqref{proof::unique::end} can be repeated with $r = R$, which yields
    \begin{equation*}
        \|\phi(t,\cdot)^2\|_{L^\infty(\Omega)} \leq 0,
    \end{equation*}
    for all $t \in [0, T]$.
\end{proof}

Combining \cref{lem::existence} and \cref{theorem::uniqueness}, we obtain the following statement.

\begin{corollary}
    Let $p \in (1, \infty]$, $\phi_0 \in L^\infty(\Omega)$ and $\veli \in  L^1((0,T), \mathrm{BV}_0(\Omega))^d \cap L^\infty((0,T)\times \Omega)^d $ with $$\mathrm{div}(\veli) \in L^p((0,T), L_{\exp}(\Omega)).$$ Then the transport equation \eqref{eq::transport} has a unique weak renormalized solution $$\phi \in C([0,T], L^\infty(\Omega)\mathrm{-w}^\ast).$$ Furthermore,
    \begin{equation*}
        \|\phi (t, \cdot)\|_{L^\infty(\Omega)} \leq \| \phi_0 \|_{L^\infty(\Omega)}
    \end{equation*}
    for any $t \in [0,T]$ and the vector field $\veli$ has the renormalization property.
    \label{corollary::uniqueness}
\end{corollary}

\subsection{Stability}
\label{subsec::stability}

Let, for $p \in (1,\infty]$ the space $B$ be given by
\begin{equation*}
    B\coloneqq \lbrace \veli \in  L^1((0,T), \mathrm{BV}_0(\Omega)^d) \cap L^\infty((0,T)\times \Omega)^d \,:\, \mathrm{div}(\veli) \in L^p((0,T), L_{\exp}(\Omega))\rbrace.
\end{equation*}
We present a generalization of the stability result \cite[Thm.~4.1]{jarde2019existence} based on the uniqueness result \cref{theorem::uniqueness}, which follows with the same argumentation as in \cite[Thm.~4.1]{jarde2019existence} building on \cref{corollary::uniqueness} instead of \cite[Thm.~2.1]{jarde2019existence}. For more details we refer to \cref{proof::sec}.

\begin{proposition}
    Let $\veli \in B$ and let the initial value satisfy $\phi_0 \in L^\infty(\Omega)$. Furthermore, let $(\veli^n)_{n \in \mathbb N} \subset B$ and $(\phi_{0,n})_{n \in \mathbb N} \subset L^\infty(\Omega)$ be two sequences with the following properties:
    \begin{enumerate}
        \item $(\phi_{0,n})_{n \in \mathbb N}$ is bounded in $L^\infty(\Omega)$ and converges to $u_0$ in $L^1(\Omega)$,
        \item  $(\veli^n)_{n \in \mathbb N}$ converges strongly to $\veli$ in $L^1((0,T)\times \Omega)^d$, and
        \item
            $(\mathrm{div}(\veli^n))_{n \in \mathbb N}$ converges strongly to $\mathrm{div}(\veli)$ in $L^1((0,T) \times \Omega)$.
    \end{enumerate}
    Then for each $1 \leq p < \infty$, the sequence of unique solutions $(\phi_n)_{n \in \mathbb N} \subset C([0,T], L^\infty(\Omega) \mathrm{-w}^\ast)$ of the transport equation \eqref{eq::transport} with vector fields $\veli^n$ and initial data $\phi_{0,n}$ is a subset of $C([0,T], L^p(\Omega))$ and converges in $C([0,T], L^p(\Omega))$ to the unique solution $\phi \in C([0,T], L^p(\Omega))$ of \eqref{eq::transport} with vector fields $\veli$ and initial value $\phi_0$.
    \label{theorem::stability}
\end{proposition}

\begin{remark}
    Under the stronger requirement $\mathrm{div}(\veli) \in L^p((0,T), L^\infty(\Omega))$ on the regularity of the divergence, a generalization is presented in \cite{jarde2019existence}, which replaces \textup{(ii)} in \cref{theorem::stability} with
    \begin{itemize}
        \item[\textup{(ii)}$^\prime$] $(\veli^n)_{n \in \mathbb N}$ is bounded in $L^q((0,T), \mathrm{BV}_0(\Omega))^d$ for some $q>1$ and $\veli^n \rightharpoonup \veli$ in $L^1((0,T)\times \Omega$.
    \end{itemize} Moreover, in Section~5 of \cite{jarde2019existence} another stability result with weaker regularity assumptions is presented. Investigating if our considerations can be adapted to these settings is left for future research.
\end{remark}

\section{Theoretical analysis of image registration problem}
\label{sec::wellposed}

In case we consider stationary vector fields,
one possible modeling choice for image registration is given as follows.

\begin{lemma}
    Let $\Omega \subset \mathbb R^d$ be a bounded Lipschitz domain,
    $W = W^{1,1}(\Omega)^d$, $X = H^{1 +\sigma}(\Omega)^d$ for $\frac12 > \sigma > 0$, $Y = W^{1,\exp}_0 (\Omega)^d$, $Z = W^{1,\infty}_0(\Omega)^d$, $\psi_\gamma(\cdot) = \beta \Psi_{\gamma}(\cdot)$, $\psi_0 (\cdot) = \beta \| \cdot \|_{W^{1,\infty}_0(\Omega)^d}$, $\eta(\cdot) = \frac\alpha2 \| \cdot \|_X^2$ for $\alpha, \beta, \gamma > 0$. Moreover, let $$j(\cdot)\coloneqq J(\Phi(\cdot, T), \phi_{\mathrm{tar}})$$ be continuous with $J(\cdot, \phi_{\mathrm{tar}}): L^p(\Omega) \to \mathbb R$ for some $p \in [1, \infty)$ and bounded from below.  Then properties \ref{assumption:1}--\ref{assumption:1.3} and \ref{assumption.continuity}--\ref{assumption:limgamma} are satisfied.
    \label{lemma::imreg::prop}
\end{lemma}
\begin{proof}
    By the choice of function spaces, \ref{assumption:1}, \ref{assumption:1.2} and \ref{assumption:1.3} are fulfilled. Due to \cref{lemma::5} and the considerations in appendix \ref{predual}, assumption \ref{assumption:1.1} holds. \ref{assumptions:weakastconv} is valid due to \cref{lem::lsc_exp}. \ref{assumption.continuity} holds due to \cref{remark::continuity}. The properties of the norm on the Hilbert space $H^{1+\sigma}(\Omega)^d$ yield \ref{assumptions:weakconv} and \ref{assumption:etacont}. By the choice of $j$ and \cref{theorem::stability}, \ref{assumption:4} holds. The estimate \eqref{eq::proof::2} and the choices of $j$, $\psi$ and $\eta$ also ensure boundedness from below on $X \cap Y$, i.e., property \ref{assumption:13} is valid. \ref{assumption:14} holds due to \cref{lem::seqweak_mod}. \Cref{corollary::estimpsi} yields \ref{assumption:estgamma}. \ref{assumption:limgamma} holds due to \eqref{eq::limit}. \ref{limit.assumption} holds due to the following considerations: Let $(v^k)_{k \in \mathbb N} \subset X \cap Y$ be a sequence that converges weakly to $\hat v$ in $X$ and weakly$^\ast$ to $\bar v$ in $Y$. For $\varphi \in C_c^\infty(\overline \Omega)^d$, we know that $\Phi(\cdot) \coloneqq \int_\Omega \langle \varphi (x), \cdot(x)\rangle \mathrm{d} x$ is an element of $X^\ast$, and, hence $\int_\Omega \langle \varphi(x), \vel^k(x) \rangle \mathrm{d}x \to \int_\Omega \langle \varphi(x), \hat \vel(x) \rangle \mathrm{d}x$.
    In addition, due to the definition of weak$^\ast$ topology, $\Phi: Y \to \mathbb R$ is continuous w.r.t. the  weak$^\ast$ topology. Therefore, $\int_\Omega \langle \varphi(x), \vel^k(x) \rangle \mathrm{d}x \to \int_\Omega \langle \varphi(x), \bar \vel(x) \rangle \mathrm{d}x $. Combining the limits, yields $\int_\Omega \langle \varphi(x), \bar \vel(x) - \hat \vel(x) \rangle \mathrm{d}x = 0$. Since $\varphi \in C_c^\infty(\overline \Omega)^d$ was chosen arbitrarily, we know due to the fundamental lemma of the calculus of variations that $\bar \vel = \hat \vel$ a.e.
\end{proof}

Combining \cref{theorem::full,lemma::imreg::prop} yields the following.

\begin{corollary}[well-posedness of relaxed image registration problem]
    Let $\Omega \subset \mathbb R^d$ be a bounded domain, $\sigma \in (0,\frac12)$, $\alpha,\beta, \gamma > 0$, and $J(\cdot, \phi_{\mathrm{tar}}): L^p(\Omega) \to \mathbb R$ be continuous for some  $p \in [1,\infty)$ and bounded from below. Moreover, let $\phi_0, \phi_{\mathrm{tar}} \in L^\infty(\Omega)$. Then the optimization problem
    \begin{equation}
        \begin{aligned}
            \min_{v \in H^{1+\sigma}(\Omega)^d \cap W^{1,{\exp}}(\Omega)^d} &J(\phi(\cdot , T), \phi_{\mathrm{tar}}) + \beta \Psi_\gamma(v) + \frac{\alpha}2 \| v \|_{H^{1+\sigma}(\Omega)^d}^2 \\
            &\begin{aligned}
                \text{s.t. }  \partial_t \phi + \vel \cdot \nabla \phi &= 0 && \text{in } \Omega \times (0,T), \\
                \phi(\cdot, 0) &= \phi_0 && \text{on } \Omega,
            \end{aligned}
        \end{aligned}
    \end{equation}
    has a solution.
\end{corollary}

Due to \cref{remark::discont}, the regularization term is not differentiable on $W^{1,\exp}(\Omega)^d$; it is not even continuous. However, discretizations that yield functions in $W^{1,\infty}(\Omega)^d$ allow for a way to circumvent this, see \cref{subsec::differentiability}. In order to verify \ref{assumption::discretization}, for any $w \in H^{1+\sigma}(\Omega)^d \cap W_0^{1,\infty}(\Omega))^d$, we construct a sequence of discretized functions that converges to $w$ w.r.t. the strong topology in $H^{1+\sigma}(\Omega)^d$ in such a way that the $W^{1,\infty}_0(\Omega)^d$-norm stays bounded.

\begin{lemma}
    Let $\Omega \subset \mathbb R^d$, $d \in \lbrace 1, 2, 3\rbrace$, be a bounded domain with smooth boundary.
    Let $X_k$ be the Lagrange $\mathcal P_1$ finite element space on the discretization $\Omega_{\frac1k}$ with mesh size smaller than or equal to $\frac1k$. Moreover, let $0 < \sigma < \frac12$, $X = H^{1+\sigma}(\Omega)^d$ and $Z = W_0^{1,\infty}(\Omega)^d$. Then \ref{assumption::discretization} is satisfied.
    \label{p16}
\end{lemma}

\begin{proof}
    Let $\epsilon > 0$ be arbitrary but fixed, and $w \in H^{1+\sigma}(\Omega)^d \cap W_0^{1,\infty}(\Omega)^d$. We first extend $w$ to a function $\tilde w \in H^{1 + \sigma}(\mathbb R^d)^d \cap W^{1,\infty}(\mathbb R^d)^d$ by setting $\tilde w(x) = w(x)$ for all $x \in \overline{\Omega}$ and $\tilde w(x) = 0$ for $x \in \mathbb R^d\setminus \overline{\Omega}$ (\cite[Thm.~11.4]{lions2012non}). Mollification with $\varphi_\eta (x) \coloneqq \eta^{-d} \varphi(\frac{x}\eta)$, where $\varphi: \mathbb R^d \to \mathbb R$ being compactly supported in $B_1(0)$, with $\int_{\mathbb R^d} \varphi(x) \mathrm{d}x = 1$, and $\lim_{\eta \to 0^+} \varphi_\eta (x) = \delta(x)$ then yields an approximation $\tilde w_\eta (x) \in C_c^\infty(\mathbb R^d)$ that converges to $\tilde w$ w.r.t. to the $H^{1+\sigma}(\mathbb R^d)^d$-norm for $\eta \to 0^+$. Therefore, we can choose $\eta = \eta(\epsilon) \leq \epsilon$ such that \begin{equation}\| \tilde w(x) - \tilde w_{\eta(\epsilon)}(x)\|_{H^{1+\sigma}(\mathbb R^d)^d} \leq \epsilon.
    \label{proof::estim1}\end{equation} Moreover, mollification does not increase the Lipschitz constant, i.e., \begin{equation*}| \tilde w_\eta |_{W^{1,\infty}(\mathbb R^d)^d} \leq | \tilde w |_{W^{1,\infty}(\mathbb R^d)^d} \leq \| w \|_{W^{1,\infty}_0(\Omega)^d},\end{equation*} where 
    $ | \cdot |_{W^{1,\infty}(\mathbb R^d)^d}$ is defined analogously to \eqref{w1inftynorm} on $\mathbb R^d$ instead of $\Omega$. Let $\hat w(x) = E(\tilde w_{\eta(\epsilon)})$, where $E: H^{1+\mu}(\mathbb R^d\setminus \Omega)^d  \to H^{1+ \mu}(\Omega)^d$, for all $\mu > 0$, denotes a continuous extension operator such that $E(u)\vert_{\partial \Omega} = u\vert_{\partial \Omega}$, $E(u) \in C^\infty_c(\Omega)$
    and there exists a constant $C>0$ such that
    \begin{equation*}
        | E(u) |_{W^{1,\infty}(\Omega)^d} \leq C \| u \|_{W^{1,\infty}(\mathbb R^d \setminus \Omega)^d}
    \end{equation*}
    for all $u \in W^{1,\infty}(\mathbb R^d \setminus \Omega)^d$
    (e.g., by reflecting $u\vert_{\mathbb R^d\setminus \Omega}$ at the boundary along the fibers of the tubular neighborhood and multiplying by a function that is in $C^\infty(\Omega)$, $1$ on $\partial \Omega$ and compactly contained in the tubular neighborhood of $\partial \Omega$). Here, $| \cdot |_{W^{1,\infty}(\Omega)^d}$, defined as \eqref{w1inftynorm}, denotes a seminorm on $W^{1,\infty}(\Omega)^d$. Due to continuity of the extension operator, \eqref{proof::estim1}, and the fact that $\tilde w\vert_{\mathbb R^d \setminus \Omega } = 0$, there exists a constant $\tilde C$ such that $\| \hat w \|_{H^{1 + \sigma}( \Omega)^d} \leq \tilde C \epsilon$, and $| \hat w |_{W^{1,\infty}(\Omega)^d} \leq \tilde C \| w \|_{W^{1,\infty}_0(\Omega)^d}$.  Define $\check w \coloneqq \tilde w_{\eta(\epsilon)} - \hat w $, which is an element of $W^{1,\infty}_0(\Omega)^d$ and bounded $W^{1,\infty}_0$-norm by a constant that is independent of $\epsilon$.
    Due to smoothness of $\check w$ and \cite[Thm.~2.6]{belgacem2001some} there exists $\ell_\epsilon \in \mathbb N$ such that $\| \Pi_\frac1{\ell_\epsilon}^1 (\check w) - \check w \|_{H^{1+ \sigma}(\Omega)^d} \leq \frac\epsilon2$, where $\Pi_h^1$ denotes the nodal interpolation operator for mesh size $h$. Moreover, by definition 
    \begin{equation*}
        \| \Pi_\frac1{\ell_\epsilon}^1(\check w)\|_{W^{1,\infty}(\Omega)^d} \leq \| \check w \|_{W^{1,\infty}_0(\Omega)^d} \leq \check C \| w \|_{W^{1,\infty}_0(\Omega)^d}
    \end{equation*} for a constant $\check C > 0$. Combining the estimates shows that \ref{assumption::discretization} is satisfied.
\end{proof}

The results derived in \cref{sec::approxnorm}--\cref{sec::wellposed} show that the conceptual framework presented in \cref{sec::general_framework} can be applied to an optical flow based optimal control formulation for image registration. In particular, due to \cref{lemma::imreg::prop} and \cref{p16} the properties \ref{assumption:1}--\ref{assumption:limgamma} are satisfied. Hence, the subsequent statement follows directly from \cref{theorem::convergence}.

\begin{corollary}
    Let $\Omega \subset \mathbb R^d$ be a bounded domain with smooth boundary.
    Moreover, for $k \in \mathbb N$, let $X_k$ be the Lagrange $\mathcal P_1$ finite element space on the discretization $\Omega_{\frac1k}$ with mesh size smaller than or equal to $\frac1k$. Let $\sigma \in (0,\frac12)$, $\alpha, \beta > 0$, and $J(\cdot, \phi_{\mathrm{tar}}) : L^p(\Omega)^d \to \mathbb R$ be continuous for some $p \in [1,\infty)$ and bounded from below. Let $j(v)\coloneqq J(\phi(\cdot, T), \phi_{\mathrm{tar}})$, with $\phi = \phi(v)$ being the solution of
    \begin{equation*}
        \begin{aligned}
            \partial_t \phi + \vel \cdot \nabla \phi &= 0 && \text{in } (0,T) \times \Omega, \\
            \phi(\cdot, 0) &= \phi_0 &&\text{on } \Omega.
        \end{aligned}
    \end{equation*}
    Moreover, let $\phi_0, \phi_{\mathrm{tar}} \in L^\infty(\Omega)$. For $\gamma > 0$, let $v_\gamma^k$ be defined as a global solution of
    \begin{equation}
        \min_{v \in X_k} j(v) + \beta \Psi_{\gamma}(v) + \frac{\alpha}2 \| v \|_{H^{1+\sigma}(\Omega)^d}^2.
    \end{equation}
    Then there exists a sequence $(\gamma_k)_{k \in \mathbb N}$ with $\gamma_k >0 $ for all $k \in \mathbb N$ and $\lim_{k \to \infty} \gamma_k = 0$ such that a subsequence of $(v_{\gamma_k}^k)_{k \in \mathbb N}$ converges (w.r.t. the strong topology in $W^{1,1}(\Omega)^d$, the weak topology in $H^{1+\sigma}(\Omega)^d$, and the weak$^\ast$-topology in $W^{1,\exp}(\Omega)^d$) to $\bar v$, where $\bar v$ is a global solution of
    \begin{equation*}
        \min_{v \in W^{1,\infty}_0(\Omega)^d \cap H^{1+\sigma}(\Omega)^d} j(v) + \beta \| v\|_{W^{1,\infty}_0(\Omega)^d} + \frac{\alpha}2 \| v \|_{H^{1+\sigma}(\Omega)^d}^2.
        \label{originalprob}
    \end{equation*}
\end{corollary}

\section{Conclusion}

We have considered the image registration problem \eqref{opt::prob} and a sequence of relaxed optimization problems that approximate the solution of \eqref{opt::prob}. We introduced a conceptual framework, for which we have proven well-posedness of the relaxation, approximation by a sequence of semidiscretized problems, and convergence w.r.t. the relaxation parameter $\gamma$. Moreover, we have verified that the conceptual framework is applicable to image registration by collecting results on Orlicz spaces, establishing existence, uniqueness and stability results for the linear hyperbolic transport equation, and discussing the relaxation of a $W^{1,\infty}_0$-norm.

\section*{Acknowledgements}

Johannes Haubner would like to thank Johannes Milz for helpful discussions, in particular on the conceptual framework, the approximation of the $L^\infty$-norm, and the use of Orlicz spaces in statistics.

\appendix

\section{Proof of \texorpdfstring{\cref{lem::existence}}{existence result}}
\label{sec::proofexistence}

We give the proof of \cref{lem::existence}.
Let $(\vel^n)_{n \in \mathbb N} \subset C^\infty( (0,T) \times \Omega)^d$ be a sequence of vector fields such that $\vel^n \to \vel$ in $L^1((0,T) \times \Omega)^{d}$ and $\mathrm{div}(\vel^n) \to \mathrm{div}(\vel) \in L^1((0,T) \times \Omega)$. Then, according to \cite[Theorem 2.1]{jarde2019existence} (which collects results from \cite{crippa2008flow, crippa2014initial}), the transport equation \eqref{eq::transport} with vector field $\vel^n$ and initial condition $\phi_0$ attains a unique solution $\phi_n \in C([0,T], L^\infty(\Omega) \mathrm{-w}^\ast)$, which fulfills the weak form of the transport equation and
\begin{align*}
    \| \phi_n (t, \cdot) \|_{L^\infty(\Omega)} \leq \| \phi_0 \|_{L^\infty(\Omega)}
\end{align*}
for each $t \in [0,T]$ and, hence, 
\begin{align*}
    \| \phi_n \|_{L^\infty((0,T) \times \Omega)} \leq \| \phi_0 \|_{L^\infty(\Omega)}.
\end{align*}
Using Banach--Alaoglu and the separability of the pre-dual $L^1((0,T) \times \Omega)$ of $L^\infty((0,T) \times \Omega)$, we know that there exists a $L^\infty$-weak$^\ast$ converging subsequence $(\phi_n)_{n \in K}$, $K\subset \mathbb N$, with limit $\bar \phi \in L^\infty((0,T) \times \Omega)$, which fulfills 
\begin{align}
    \| \bar \phi \|_{L^\infty((0,T) \times \Omega)} \leq \| \phi_0 \|_{L^\infty(\Omega)}.
    \label{proof::eq::1}
\end{align}
Moreover, due to limit considerations of the weak forms involving $\phi_n$ and $\vel^n$ for $n \in K$, $(\phi_n)_{n \in K}$ converges $L^\infty$-weakly$^\ast$ to a solution $\phi$ of the transport equation. Hence $\bar \phi = \phi$ a.e.
Assume that there exist $\epsilon > 0$ and $\tau \in [0,T]$ with 
\begin{align*}\| \phi (\tau, \cdot)\|_{L^\infty(\Omega)} \geq \| \phi_0 \|_{L^\infty(\Omega)} + \epsilon. \end{align*} 
Then, there exists $\psi \in L^1(\Omega)$ with $\| \psi \|_{L^1(\Omega)} = 1$ such that 
$\int_\Omega \phi(\tau, x) \psi (x) \mathrm{d}x \geq \| \phi_0 \|_{L^\infty(\Omega)} + \frac{\epsilon}2$ and due to $\phi \in C([0,T], L^\infty(\Omega)\mathrm{-w}^\ast)$, there exist $\delta > 0$ and $r,s \in [0,T]$ with $r < s$, $| r - s| = \delta$ and $\tau \in [r,s]$ such that 
\begin{align}\int_\Omega \phi(t, x) \psi (x) \mathrm{d}x \geq \| \phi_0 \| + \frac{\epsilon}4
\label{proof::eq::2}
\end{align}
for all $t \in [r, s]$. 
Hence, since $\| \frac{1}{s-r} 1_{[r,s]} \psi \|_{L^1((0,T)\times \Omega)} = 1$, we obtain with \eqref{proof::eq::2} that
\begin{align*}
    \| \phi \|_{L^\infty((0,T)\times\Omega)} &\geq \int_0^T \int_\Omega \frac{1}{s-r} 1_{[r,s]}(t) \psi(x) \phi(t,x) \mathrm{d}x \mathrm{d}t\\
    & = \frac{1}{s-r} \int_r^s \int_\Omega \psi(x) \phi(t,x) \mathrm{d}x \mathrm{d}t \\
    & \geq \| \phi_0 \| + \frac{\epsilon}4,
\end{align*}
which contradicts \eqref{proof::eq::1}.

\section{Proof of \texorpdfstring{\cref{theorem::stability}}{stability theorem}}
\label{proof::sec}

For the sake of completeness, we give a proof of \cref{theorem::stability} based on \cite{jarde2019existence}.
Let $n \in \mathbb N$ and $\phi_n \in C([0,T], L^\infty(\Omega)\mathrm{-w}^\ast)\cap L^\infty((0,T)\times\Omega)$ be the solution of \eqref{eq::transport} with vector field $\veli^n$ and initial value $\phi_{0,n}$. Due to \cref{corollary::uniqueness}, we know that there exists a constant $C_1 > 0$ such that $(\phi_n)_{n \in \mathbb N}$
fulfills
\begin{equation*}
    \| \phi_n(t, \cdot) \|_{L^\infty(\Omega)} \leq \| \phi_{0,n} \|_{L^\infty(\Omega)} \leq C_1
\end{equation*}
for all $n \in \mathbb N$ and $t \in [0,T]$.
Hence, an arbitrary subsequence has a convergent subsequence $(\phi_n)_{n \in K}$, $K \subset \mathbb N$, with weak$^\ast$ limit $\phi \in L^\infty(\Omega \times (0,T))$. Since $\phi_n$ fulfills the weak formulation \eqref{eq::weak_form} with vector field $b_n$ and initial condition $\phi_{0,n}$, it follows from $L^\infty((0,T)\times \Omega)^d$-weak$^\ast$ convergence of $(\phi_n)_{n \in K}$, and strong $L^1((0,T)\times \Omega)^d$ convergence of $(\veli^n)_{n \in K}$ and $(\mathrm{div}(\veli^n))_{n \in K}$, and strong $L^1(\Omega)$-convergence of $(\phi_{0,n})_{n \in K}$ that the limit $\phi$ fulfills \eqref{eq::weak_form} with vector field $\veli$ and initial condition $\phi_0$. Due to uniqueness (\cref{theorem::uniqueness}) this limit is independent of the choice of the subsequence.

In order to establish the stated regularity and convergence results we consider
\begin{equation*}
    g_{n,\varphi} (t) \coloneqq \langle \phi_n(t, \cdot), \varphi \rangle
\end{equation*}
for $\varphi \in C_c^\infty(\Omega)$.
First of all, let $p = 2$ (and afterwards use this result to generalize to $1 \leq p < \infty$).
With the above considerations and due to uniqueness of the solution of \eqref{eq::weak_form} with vector field $b$ and initial condition $u_0$, a standard proof by contradiction shows that
\begin{equation}
    \lim_{n \to \infty} g_{n, \varphi}(t, \cdot) = \langle \phi(t, \cdot), \varphi \rangle
    \label{eq::stability::proof::limit}
\end{equation}
for almost every $\varphi \in L^2(\Omega)$ and a.e. $t \in (0,T)$.
In order to apply Arzel\'a--Ascoli, analogously to \cite{jarde2019existence}, we show uniform equicontinuity of $\lbrace g_{n,\varphi}\,:\, n \in \mathbb N\rbrace$ for arbitrary $\varphi \in L^2(\Omega)$, i.e.
\begin{equation}
    \forall \epsilon > 0 \ \exists \delta > 0 \ \text{such that } d(t_1, t_2) < \delta \Rightarrow \ | g_{n, \varphi}(t_1, \cdot) - g_{n,\varphi} (t_2, \cdot) | < \epsilon \quad \forall n \in \mathbb N.
    \label{eq::uniform_equicontinuity}
\end{equation}
Since $\phi_n$ fulfills \eqref{eq::weak_form} with vector field $b_n$ and initial condition $\phi_{0,n}$, we know for arbitrary $\psi \in C_c^\infty((0,T))$ that
\begin{equation*}
    \begin{aligned}
        \int_0^T \psi(t) \frac{d}{d t} g_{n, \varphi} (t, \cdot) \mathrm d t &= - \int_0^T \psi^\prime(t) \langle \phi_n(t, \cdot), \varphi \rangle \mathrm{d}t\\
        &= \int_0^T \psi(t) [ \langle \phi_n(t, \cdot) \veli^n(t, \cdot), \nabla \varphi \rangle + \langle \phi_n(t, \cdot) \mathrm{div}(\veli^n(t, \cdot)), \varphi \rangle] \mathrm d t,
    \end{aligned}
\end{equation*}
i.e.~$(g_{n,\varphi})_{n \in \mathbb N}$ is weakly differentiable with derivative
\begin{equation*}
    g_{n, \varphi}^\prime (t, \cdot) = \langle \phi_n(t, \cdot) \veli^n(t, \cdot), \nabla \varphi\rangle + \langle \phi_n(t, \cdot) \mathrm{div}(\veli^n(t, \cdot)), \varphi \rangle.
\end{equation*}
Hence, for a sequence $(\varphi_k)_{k \in \mathbb N} \subset C_c^\infty(\Omega)$ converging to $\varphi \in L^2(\Omega)$ and $0 \leq t_1 < t_2 \leq T$, we obtain
\begin{equation}
    \begin{multlined}[t]
        |g_{n,\varphi} (t_1, \cdot) - g_{n, \varphi} (t_2, \cdot)| \\
        \begin{aligned}
            & \leq |g_{n, \varphi} (t_1, \cdot) - g_{n, \varphi_k} (t_1, \cdot) | + | g_{n, \varphi_k} (t_1, \cdot) - g_{n, \varphi_2} (t_2, \cdot) | +
            | g_{n, \varphi k} (t_2, \cdot) - g_{n, \varphi} (t_2, \cdot)| \\
            & = |\langle \phi_n(t_1, \cdot), \varphi - \varphi_k \rangle| + |\langle \phi_n(t_2, \cdot), \varphi - \varphi_k \rangle| + |\int_{t_1}^{t_2} g_{n, \varphi_k}^\prime (t, \cdot) \mathrm{d}t| \\
            & \leq ( \| \phi_n(t_1, \cdot)\|_{L^2(\Omega)} + \| \phi_n(t_2, \cdot)\|_{L^2(\Omega)}) \| \varphi_k - \varphi \|_{L^2(\Omega)} + \int_{t_1}^{t_2} |g^\prime_{n, \varphi_k} (t)| \mathrm{d}t.
        \end{aligned}
    \end{multlined}
    \label{eq::stability::proof}
\end{equation}
Moreover, we have the estimate
\begin{equation}
    \begin{aligned}[t]
        \int_{t_1}^{t_2} |g^\prime_{n, \varphi_k} (t)| \mathrm{d}t  &\leq C_1 \| \varphi_k \|_{H^1(\Omega)} ( \| \veli^n(t, \cdot) \|_{L^1(\Omega)^d} + \| \mathrm{div}(\veli^n(t, \cdot))\|_{L^1(\Omega)}) | t_2 - t_1|
        \\ &\leq C_2 C_1 \| \varphi_k \|_{H^1(\Omega)} | t_2 - t_1|
    \end{aligned}
    \label{eq::stability::proof2}
\end{equation}
with $C_2 \coloneqq \sup_{n \in \mathbb N} ( \| \veli^n(t, \cdot) \|_{L^1(\Omega)^d} + \| \mathrm{div}(\veli^n(t, \cdot))\|_{L^1(\Omega)})$, which is bounded due to properties (ii) and (iii). For arbitrary $\epsilon > 0$, there exists $k_\epsilon \in \mathbb N$ such that $\| \varphi_{k_\epsilon} -\varphi \|_{L^2(\Omega)} \leq \epsilon$. Furthermore, $\delta(\epsilon) > 0$ can be chosen such that $\int_{t_1}^{t_2} | g_{n,\varphi_{k_\epsilon}}^\prime (t) | \mathrm{d}t \leq \epsilon$ if $|t_2 - t_1 | \leq \delta(\epsilon).$ This, together with \eqref{eq::stability::proof} and \eqref{eq::stability::proof2} yields
\begin{equation*}
    | g_{n, \varphi}(t_1) - g_{n,\varphi} (t_2) | \leq (2 |\Omega|^\frac12 C_1 +1) \epsilon
\end{equation*}
for all $n \in \mathbb N$ if $|t_2 - t_1| \leq \delta(\epsilon)$, which implies \eqref{eq::uniform_equicontinuity}.

Arzel\'a--Ascoli yields that there exists a subsequence $(g_{n, \varphi})_{n \in M}$, $M \subset \mathbb N$ and $g \in C([0,T])$ such that $\lim_{M \ni n \to \infty} g_{n,\varphi} = g$. Combining this finding with \eqref{eq::stability::proof::limit} yields $g = \langle \phi(t, \cdot), \varphi \rangle$ for arbitrary $\varphi \in L^2(\Omega)$. Since $g$ is uniquely defined, we can again argue with a contradiction proof to show that $\lim_{n \to \infty} g_{n,\varphi} = g$. Therefore, $\phi \in C([0,T], L^2(\Omega)-w)$ and $ \phi_n \to \phi $ in $C([0,T], L^2(\Omega)-w)$.

Due to the renormalization property of $\veli$, following the previous argument, one can show that $((\phi_n)^2)_{n \in \mathbb N}$ converges to $\phi^2$ in $C([0,T], L^2(\Omega)-w)$. \cite[Lem.~4.2]{jarde2019existence} then implies that, for all $n \in \mathbb N$, $\phi_n, \phi \in C([0,T],  L^2(\Omega))$ and $\phi_n \to \phi$ in $C([0,T], L^2(\Omega))$.

As in \cite[Proof of Thm.~4.1]{jarde2019existence}, it can be concluded that the result holds for general $1 \leq p < \infty$. The case $1 \leq p \leq 2$ is covered due to the continuous embedding of $C([0,T], L^2(\Omega))$ into $C([0,T], L^p(\Omega))$ for $1 \leq p \leq 2$. For $2< p < \infty$ and $t, s \in [0,T]$, it holds that
\begin{equation*}
    \| \phi_n(t, \cdot) - \phi_n(s, \cdot) \|_{L^p(\Omega)}^p \leq (2 C_1)^{p-2} \| \phi_n (t, \cdot) - \phi_n(s, \cdot) \|_{L^2(\Omega)}^2
\end{equation*}
and, for $t \in [0,T]$
\begin{equation*}
    \| \phi_n(t, \cdot) - \phi(t, \cdot) \|_{L^p(\Omega)}^p \leq (2 C_1)^{p-2} \| \phi_n(t, \cdot) - \phi(t, \cdot) \|_{L^2(\Omega)}^2.
\end{equation*}
Taking the supremum over $[0,T]$ yields the statement.

\section{Predual of \texorpdfstring{$\scriptstyle W^{1,\exp}_0(\Omega)$}{W1exp}}
\label{predual}

We define
\begin{equation*}
    W^{-1, L\log{}L}(\Omega) \coloneqq \lbrace f_0 + \sum_{i=1}^d \partial_{x_i} f_i \in \mathcal D^\prime (\Omega): ~ f_i \in L\log{}L(\Omega) ~ \forall i \in \lbrace 0, \ldots, d \rbrace \rbrace
\end{equation*}
as a subspace of the space of distributions $\mathcal D^\prime(\Omega)$ on $\Omega$ equipped with the norm
\begin{equation*}
    \| f \|_{W^{-1, L\log{}L}(\Omega)} \coloneqq \inf \left \lbrace \sum_{i=0}^d \|f_i \|_{L\log{}L(\Omega)}\,:\, f_i \in L\log{}L(\Omega) ~ \forall i \in \lbrace 0, \ldots, d \rbrace, ~ f = f_0 + \sum_{i=1}^d \partial_{x_i} f_i \right\rbrace.
\end{equation*}
By definition, there exists a bounded linear surjection
\begin{equation*}
    \iota: L\log{}L(\Omega, \mathbb R^{d+1}) \to W^{-1, L\log{}L}(\Omega), \quad \iota((f_0, \ldots, f_d)^\top) \coloneqq f_0 + \sum_{i=1}^d \partial_{x_i} f_i.
\end{equation*}
Moreover, there exists a bounded linear right inverse
\begin{equation*}
    \iota^{-1}: W^{-1, L\log{}L}(\Omega) \to L\log{}L(\Omega, \mathbb R^{d+1})
\end{equation*}
such that $\iota \circ \iota^{-1}: W^{-1, L\log{}L}(\Omega) \to W^{-1, L\log{}L}(\Omega)$ is the identity on $W^{-1, L\log{}L}(\Omega)$.
We consider the standard distributional pairing $\langle \cdot \rangle: ~ \mathcal D^\prime(\Omega) \times C_c^\infty(\Omega) \to \mathbb R$. In the following we show that this pairing can be extended to a pairing that maps from $W^{1, \exp}(\Omega) \times W^{-1, L\log{}L}(\Omega)$ to $\mathbb R$, which is given as
\begin{equation}
    \langle\!\langle \phi, f \rangle\!\rangle \coloneqq \int_\Omega \phi(x) (\iota^{-1}(f))_0(x) \mathrm{d} x - \sum_{i=1}^n \int_\Omega \partial_{x_i} \phi(x) (\iota^{-1}(f))_i (x)\mathrm{d} x.
    \label{eq::pairing}
\end{equation}
Moreover, we show that the dual space of $W^{-1, L\log{}L}(\Omega)$ can be identified with $W^{1, \exp}(\Omega)$.

We first consider well-definedness of \eqref{eq::pairing} and show that  $W^{-1, L\log{}L(\Omega)}$ is a subspace of the pre-dual of $W^{1,\exp}(\Omega)$:
Since $C_c^\infty(\Omega)$ is dense in $L\log{}L(\Omega)$, by its definition, $\iota(C_c^\infty(\Omega))$ is dense in $W^{-1, L\log{}L}(\Omega)$ (and therefore $W^{-1, L\log{}L}(\Omega)$ is also separable). Let $\phi \in W^{1,\exp}(\Omega)$ be arbitrary but fixed. Further, let $f \in \iota(C_c^\infty(\Omega, \mathbb R^{d+1}))$, then there exist $f_0, \ldots, f_d \in C_c^\infty(\Omega, \mathbb R^{d+1})$ such that $f = \iota((f_0, \ldots, f_d)^\top)$ and the pairing $\langle\!\langle \phi, \cdot \rangle\!\rangle$ is well-defined, bounded and linear. Integration by parts further shows that \eqref{eq::pairing} is consistent with the distributional pairing. Continuity arguments yield well-definedness, boundedness and linearity of $\langle\!\langle \phi, \cdot\rangle\!\rangle: ~W^{-1, L\log{}L}(\Omega) \to \mathbb R$.
Hence, $W^{1,\exp}(\Omega)$ can be identified with a subset of $(W^{-1, L\log{}L}(\Omega))^\ast$.

It remains to show that the dual space of $W^{-1, L\log{}L}(\Omega)$ can be identified with a subspace of $W^{1,\exp}(\Omega)$. Let $\phi \in (W^{-1, L\log{}L}(\Omega))^\ast$ be arbitrary but fixed.
By definition of the adjoint operator $\iota^\ast: (W^{-1, L\log{}L}(\Omega))^\ast \to (L\log{}L(\Omega, \mathbb R^{d+1})^\ast$, we know that, for $f \in C_c^\infty(\Omega)$,
\begin{equation*}
    \langle\!\langle \phi, f \rangle\!\rangle = \langle \phi, f\rangle = \langle \phi, \iota \circ \iota^{-1} f \rangle = \langle \iota^\ast \phi, \iota^{-1} f \rangle = \sum_{i=0}^d \int_\Omega (\iota^\ast \phi)_i (x) (\iota^{-1} f)_i (x) \mathrm{d} x,
\end{equation*}
where we identified $L\log{}L(\Omega, \mathbb R^{d+1})^\ast$ with $L^{\exp}(\Omega, \mathbb R^{d+1})$ in the latter equality (see \cref{predualLexp}).
Since the linear operators coincide on a dense subspace in $W^{-1, L\log{}L}(\Omega)$, they coincide on $W^{-1, L\log{}L}(\Omega)$. Let $f_0 \in C_c^\infty(\Omega)$, choosing $f = \partial_{x_i} f_0$ yields
\begin{equation*}
    - \int_\Omega \partial_{x_i} \phi(x) f_0 (x) \mathrm{d}x = \langle\!\langle \phi, f \rangle\!\rangle = \langle \iota^\ast \phi, \iota^{-1} f \rangle  = \int_\Omega (\iota^\ast \phi)_i (x) f_0 (x) \mathrm{d} x.
\end{equation*}
Since $f_0 \in C_c^\infty(\Omega)$ is chosen arbitrarily, $-\partial_{x_i} \phi = (\iota^\ast \phi)_i$ a.e. Therefore, $\phi \in W^{1,\exp}(\Omega)$.
This shows that $W^{1,\exp}(\Omega)$ has a separable predual.

Since $W_0^{1,\exp}(\Omega)$ is a weakly-$\ast$ closed linear subspace of $W^{1,\exp}(\Omega)$, \cite[4.6, 4.7, 4.8, 4.9]{Rudin1991} show that $W_0^{1,\exp}$ is isometrically isomorphic to $(W^{-1, L \log{}L}(\Omega) / \prescript{\perp}{}{W_0^{1,\exp}(\Omega)})^\ast$, the dual space of the quotient space $W^{-1, L \log{}L}(\Omega) / \prescript{\perp}{}{W_0^{1,\exp}(\Omega)}$, where $$\prescript{\perp}{}{W_0^{1,\exp}(\Omega)} \coloneqq \lbrace f \in W^{-1, L\log{}L}(\Omega)\,:\, \langle f, \phi \rangle = 0 ~ \forall \phi \in W_0^{1,\exp}(\Omega) \rbrace.$$ Moreover, since $W^{-1, L \log{}L}(\Omega)$ is separable, also $W_0^{1,\exp}(\Omega)$ has a separable predual.

\bibliographystyle{jnsao}
\bibliography{bib}

\end{document}